\documentclass[preprint,aos]{imsart}

\usepackage{amsmath,amsfonts,amsthm,dsfont,accents,tikz,float}
\usepackage[numbers]{natbib}
\RequirePackage[colorlinks,citecolor=green,urlcolor=blue]{hyperref}

\numberwithin{equation}{section}
\makeatletter

\newcommand{\Rmnum}[1]{\expandafter\@slowromancap\romannumeral #1@}

\def\widebar{\accentset{{\cc@style\underline{\mskip10mu}}}}
\makeatother

\newcommand{\E}{\mathbb{E}}
\newcommand{\F}{\mathcal{F}}
\newcommand{\T}{\mathrm{T}}
\newcommand{\ds}{\,\mathrm{d}}
\newcommand{\e}{\varepsilon}
\newcommand{\rr}{\gamma}
\newcommand{\R}{\mathbb{R}}
\newcommand{\s}{\mathcal{S}}
\newcommand{\I}{\mathbf{I}_d}
\newcommand{\1}[1]{\mathds{1}_{\{#1\}}}

\theoremstyle{plain}
\newtheorem{assumption}{Assumption}
\newenvironment{assump}[1]
	{\renewcommand\theassumption{#1}\assumption}
	{\endassumption}

\newtheorem{lem}{Lemma}
\newtheorem{thm}{Theorem}
\newtheorem{prop}{Proposition}
\newtheorem{corol}{Corollary}

\theoremstyle{definition}
\newtheorem{defn}{Definition}
\newtheorem*{remk}{Remark}

\graphicspath{{pics/}}

\begin{document}

\begin{frontmatter}
\title{Inference for Volatility Functionals of Multivariate It\^o Semimartingales Observed with Jump and Noise}
\runtitle{Volatility Functionals via Local Moving Averages}

\author{\fnms{Richard Y.} \snm{Chen}\thanksref{t} \ead[label=e1]{yrchen@uchicago.edu}}
\thankstext{t}{Supported by National Science Foundation under grant DMS 17-13129.}
\affiliation{University of Chicago}
\runauthor{Chen, R. Y.}
\address{Department of Statistics\\
	University of Chicago\\
	Chicago, IL 60637\\
	\printead{e1}\phantom{E-mail:\ }}

\begin{abstract}
	This paper presents the nonparametric inference for nonlinear volatility functionals of general multivariate It\^o semimartingales, in high-frequency and noisy setting. Pre-averaging and truncation enable simultaneous handling of noise and jumps. Second-order expansion reveals explicit biases and a pathway to bias correction. Estimators based on this framework achieve the optimal convergence rate. A class of stable central limit theorems are attained with estimable asymptotic covariance matrices. This paper form a basis for infill asymptotic results of, for example, the realized Laplace transform, the realized principal component analysis, the continuous-time linear regression, and the generalized method of integrated moments, hence helps to extend the application scopes to more frequently sampled noisy data.
\end{abstract}
\end{frontmatter}

\section{Introduction}\label{sec:intro}
This paper concerns statistical inference and applications for integrated volatility functionals from high-frequency data modeled by an It\^o semimartingale observed with noise. The functionals have of the form
\begin{equation}\label{def.S(g)}
	S(g)_t = \int_0^t g(c_s)\ds s,
\end{equation}
here $t$ is finite,
$g:\R^{d\times d}\mapsto \R^r$ is any three-times continuously differentiable function on some compact set, 
$c_s$ is a positive-definite matrix that is the instantaneous covariance of the continuous part of the It\^o semimartingale.

In absence of noise, inferential frameworks of volatility functional estimation were established by \cite{jr13, llx19, mz09}. Subsequently, specialized methodologies for various applications with novel empirical results blossomed in recent years, for example, \cite{lx16,ax19,ltt17}.

To cope with noise, this paper embeds the \textit{pre-averaging} method \cite{jpv10, j09, pv09b} into the general framework \cite{jr13}. In this sense, this work extends the inferential framework to accommodate noisy data, and generalizes the pre-averaging method to nonlinear transformations of volatility. On the road to a rate-optimal central limit theorem (CLT) with such generality, there are the following technicalities:
\begin{itemize}
	\item \textit{Stochastic volatility}: an nonparametric model (\ref{def.X}) is used for robustness, yet, it becomes crucial to simultaneously control statistical error (due to noise) and discretization error (attributable to evolving parameters).
	\item \textit{Noise \& Jump}: there is an interplay between noise and jump, which necessitates truncating jumps on top of local moving averages, in order to recover volatility from noisy and jumpy observations.
	\item \textit{Dependence}: because of overlapping windows in pre-averaging, the local moving averages are highly correlated to which standard CLTs does not apply. The ``big block - small block'' technique of \cite{j09} is used instead.
	\item \textit{Bias}: generally there is an asymptotic bias due to nonlinearity of $g$ in (\ref{def.S(g)}). In this paper, the bias is explicitly calculated and removed.
	\item \textit{Unbounded derivatives}: some important applications, e.g., precision matrix estimation and linear regression, correspond to $g$'s with singularities in derivatives near the origin, where the original framework \cite{jr13} does not apply. A spatial localization argument by \cite{ltt17} is called upon in conjunction with a uniform convergence result.
\end{itemize}
It is the author's sincere hope, by solving these technicalities above, this paper will be able to offer a share of contribution to push the inferential framework to a new frontier of potentials and possibilities, and lend the effort to extend the corresponding applications to adopt noisy high-frequency data where exciting new stories await.

\section{Setting}\label{sec:model}
\subsection{Model}
This paper assumes the data is generated from a process $Y$, and for any $t>0$ there is a probability transition kernel $Q_t$ linking another process $X$ to $Y$ where $X$ is a solution to the stochastic differential equation
\begin{equation}\label{def.X}
	X_t = X_0 + \int_0^tb_s\ds s + \int_0^t\sigma_s\ds W_s + J_t
\end{equation}
$b_s\in\R^d$, $\sigma_s\in\R^{d\times d'}$ with $d\le d'$ and the volatility $c_s=\sigma_s\sigma_s^\T$ is positive semidefinite, $W$ is a $d'$-dimensional standard Brownian motion, $J$ is purely discontinuous process described by (\ref{J}). 

In this model, the noisy observations are samples from $Y$, and the underlying process before noise contamination is assumed as an It\^o semimartingale. 
\begin{center}
	\usetikzlibrary{positioning}
	\usetikzlibrary{shapes,snakes}
	\begin{tikzpicture}[xscale=12,yscale=6,>=stealth]
	\tikzstyle{e}=[rectangle,minimum size=5mm,draw,thick]
	\tikzstyle{v}=[ellipse,  minimum size=5mm,draw,thick]
	\node[e] (X)   [draw=red!60] {It\^o Semimartingale $X$};
	\node[e] (Y)   [draw=blue!70,right=of X] {Noisy Process $Y$};
	\node[v] (D)   [draw=black!60,right=of Y] {Noisy Data};
	\draw[thick,->,snake=snake] (X) to node[anchor=north]{$(Q_t)$} (Y);
	\draw[thick,->] (Y) to node[anchor=north]{sample} (D);
	\end{tikzpicture}
\end{center}
An example of this model is
\begin{equation}\label{def.Y}
	Y_t = f(X_t, \e_t)
\end{equation}
where $\e$ is a white noise process and $f:\R^d\times\R^d\mapsto\R^d$ is such that the conditional mean of $Y_t$ is $X_t$. Generally, the noise model induced by $(Q_t)$ incorporates additive white noise, rounding error, the combination thereof as special cases. Besides the probabilistic structure, the inferential framework also requires additional assumptions:
\begin{itemize}
	\item the drift $b$ has a smooth trajectory in certain sense (see appendix \ref{sec:assump});
	\item the volatility $c$ is a locally spatially restricted It\^o semimartingale\footnote{However, it is important to accommodate long-memory volatility model. The volatility functional inference in long-memory and noisy setting is an open question under investigation.} such that both $c$ and $c^{-1}$ is locally bounded;
	\item $J$ may exhibit infinite activities but has finite variation, i.e., finite-length trajectory;
	\item  the noise variance is an It\^o semimartingale; conditioning on all the information on $X$, there is no autocorrelation in noise.\footnote{When the observations are mixed with colored noise, the statistical property of this methodology is unknown. Since it is empirically important, the author hopes this question can be illuminated by future research.}
\end{itemize}
These assumptions are necessary for the CLT and for applicability over functions of statistical interest. For readers interested in the precise description of the model specification and assumptions, please refer to appendix \ref{sec:assump}.

\subsection{Observations}
This work treats regularly sampled observations and considers in-fill asymptotics\footnote{aka fixed-domain asymptotics, high-frequency asymptotics, small-interval asymptotics}. Specifically, the samples are observed every $\Delta_n$ time units on a finite time interval $[0,t]$ where $n=\lfloor t/\Delta_n\rfloor$ is the sample size. As $n\to\infty$, $\Delta_n\to0$ while $t$ is fixed.

Throughout this paper, $U^n_i$ is written for $U_{i\Delta_n}$ where $U$ can be a process or filtration, for example, $c^n_i$ denotes the value of volatility $c$ at time $i\Delta_n$; for any process $U$, $\Delta^n_iU$ represents the increment $U^n_i-U^n_{i-1}$.

\subsection{Notations} For $r\in\mathbb{N}^+$, $\mathcal{C}^r(\mathcal{S})$ denotes the space of $r$-time continuously differentiable functions on the domain $\mathcal{S}$; $\mathcal{S}^+_d$ is the convex cone of $d\times d$ positive semidefinite matrices; $\|\cdot\|$ denotes a norm on vectors, matrices or tensors; given $a\in\R$, $\lfloor a\rfloor$ denotes the largest integer no more than $a$; $a\vee b=\max\{a,b\}$, $a\wedge b=\min\{a,b\}$; $a_n \asymp b_n$ means both $a_n/b_n$ and $b_n/a_n$ are bounded for large $n$; $\mathbf{A}^\mathrm{T}$ is the transpose of the vector or matrix; for a multidimensional array, the entry index is written in the superscript, e.g., $X_t=(X^1_t,\cdots,X^d_t)^\T$, $c^{jk}$ denotes the $(j,k)$ entry in the matrix $c$; $\partial_{jk}g$ and $\partial^2_{jk,lm}g$ denote the gradient and Hessian with respect to the $(j,k)$-th and $(l,m)$-th entries; $\overset{\mathcal{L}-s(f)}{\longrightarrow}$ (resp. $\overset{\mathcal{L}-s}{\longrightarrow}$) denotes stable convergence of processes (resp. variables) in law\footnote{See section 2.2.1, 2.2.2 in \cite{jp12} for stable convergence. The sampling variation of the estimator depends on the realization of the process $c$, hence we need a mode of convergence in which the estimator converges jointly with other variables, so that one can consistently estimate the asymptotic variance to compute confidence intervals.}; $\overset{u.c.p.}{\longrightarrow}$ denotes uniform convergence in probability on compact sets; $\mathcal{MN}(\cdot,\cdot)$ is a mixed Gaussian distribution.

\section{Methods}\label{sec:method}
The estimation methodology consists of 5 components:
\begin{enumerate}
\item[i.] local moving averages of noisy data by a smoothing kernel $\varphi$, 
which act as proxies for $X^n_i$'s;
\item[ii.] jump truncation operated on local moving averages;
\item[iii.] spot volatility estimator $\widehat{c}^n_i$'s for estimating $c^n_i$'s;
\item[iv.] Riemann sum of $g(\widehat{c}^n_i)$'s for approximating $\int g(c_s)\ds s$;
\item[v.] bias correction due to the nonlinearity, e.g., in case of $d=1$ and constant volatility, by Taylor expansion, the estimation error of the plug-in estimator $g(\widehat{c})$ can be decomposed as 
\begin{equation*}
	g(\widehat{c}) - g(c) = \underbrace{\partial g(c)(\widehat{c} - c)}_{\text{variance}} + \underbrace{\frac{1}{2}\partial^2 g(c)(\widehat{c} - c)^2}_{\text{bias}} + \underbrace{O_p(|\widehat{c} - c|^3)}_{\text{negligible}}
\end{equation*}
the bias arises from the quadratic form of estimation error of $\widehat{c}$, provided $g$ has a non-zero Hessian. This bias term does not affect the consistency, but one needs to explicitly correct the bias to get a CLT.
\end{enumerate}
The moving-average idea is due to \cite{j09,jpv10}; the truncation is modified from (16.4.4) in \cite{jp12}; the plug-in and bias correction are inspired by \cite{jr13}. The specific recipe is given next.

\subsection{Building blocks}
For the local moving averages, we choose a smoothing kernel $\varphi$ such that
\begin{equation}\label{phi.cond}
\begin{array}{l}
	\text{supp}(\varphi)\subset(0,1),\, \int_0^1\varphi^2(s)\ds s>0  \\
	\varphi\in\mathcal{C} \text{ is piecewise } \mathcal{C}^1;\, \varphi' \text{ is piecewise Lipschitz}
\end{array}
\end{equation}
Choose an integer $l_n$ as the number of observations in each smoothing window, define $\varphi^n_h=\varphi(h/l_n)$ and $\psi_n=\sum_{h=1}^{l_n-1}(\varphi^n_h)^2$. Associate the following quantities with a generic process $U$:
\begin{equation}\label{def.Ubar.Uhat}
\begin{array}{lcl}
	\widebar{U}^n_i &=& (\psi_n)^{-1/2}\sum_{h=1}^{l_n-1}\varphi^n_h\Delta^n_{i+h-1}U \\
	\widehat{U}^n_i &=& (2\psi_n)^{-1}\sum_{h=0}^{l_n-1}(\varphi^n_{h+1}-\varphi^n_h)^2\Delta^n_{i+h}U\cdot\Delta^n_{i+h}U^\T
\end{array}
\end{equation}
$\widebar{Y}^n_i$ is a local moving average of the noisy data $Y^n_i$'s and is a proxy for $\Delta^n_iX$, $\widehat{Y}^n_i$ serves as noise correction to $\widebar{Y}^n_i$. Based on these 2 ingredients, choose $k_n>l_n$, define the spot volatility estimator as
\begin{equation}\label{def.chat}
	\widehat{c}^n_{i} \equiv \frac{1}{(k_n-l_n)\Delta_n}\sum_{h=1}^{k_n-l_n+1}\Big(\widebar{Y}^n_{i+h}\cdot\widebar{Y}^{n,\mathrm{T}}_{i+h}\1{\|\widebar{Y}^n_{i+h}\|\le\nu_n} - \widehat{Y}^n_{i+h}\Big)
\end{equation}
where $\nu_n\asymp\Delta_n^\rho$ is a truncation threshold for jumps. The choice of $\rho$ is stated in (\ref{tuning}). A spot noise variance estimator is also needed:
\begin{equation}\label{def.rhat}
	\widehat{\gamma}^n_i \equiv \frac{1}{2m_n}\sum_{h=1}^{m_n}\Delta^n_{i+h}Y\cdot\Delta^n_{i+h}Y^\T
\end{equation} 
where $m_n=\lfloor\theta'\Delta_n^{-1/2}\rfloor$, $\theta'$ positive finite. 

\subsection{The estimator}
\begin{defn}\label{def.S(g)hat}
	Let $N^n_t=\lfloor t/(k_n\Delta_n)\rfloor$, the estimator of $(\ref{def.S(g)})$ is defined as
	\begin{equation*}
		\widehat{S}(g)^n_t \equiv k_n\Delta_n\sum_{i=0}^{N^n_t-1}\left[g(\widehat{c}^n_{ik_n}) - B(g)^n_{ik_n}\right] \times a^n_t
	\end{equation*}
	where $B(g)^n_i$ is a de-biasing term of the form
	\begin{equation*}
		B(g)^n_i = \frac{1}{2k_n\Delta_n^{1/2}}\sum^d_{j,k,l,m=1}\partial^2_{jk,lm}g(\widehat{c}^n_i)\times\Xi\big(\widehat{c}^n_i,\widehat{\rr}^n_i\big)^{jk,lm}
	\end{equation*}
	$\widehat{c}^n_i$, $\widehat{\rr}^n_i$ are defined in (\ref{def.chat}), (\ref{def.rhat}), $\Xi$ is defined in (\ref{def.Xi}), and 
	$a^n_t = t/(N^n_tk_n\Delta_n)$ is a finite-sample adjustment.\footnote{Overlapping intervals are used to compute $\widehat{c}^n_i$'s, non-overlapping intervals are used to compute $\widehat{S}(g)^n_t$. The local moving averages computed over overlapping intervals in (\ref{def.chat}) are necessary to achieve the optimal convergence rate. By contrast, overlapping intervals in $\widehat{S}(g)'^n_t\equiv\Delta_n\sum_{i=0}^{\lfloor t/\Delta_n\rfloor-1}[g(\widehat{c}^n_i) - B(g)^n_i]$ do not improve the convergence rate nor efficiency, though lead to robustness in finite sample. In fact, the overlapping-interval-based estimator has the same asymptotic result as that of $\widehat{S}(g)^n_t$ in section \ref{sec:asymp}.}
\end{defn}
Besides $\varphi$, there are 3 tuning parameters in this estimator:
\begin{table}[H]
	\centering
	\begin{tabular}{c|c|r|l}
		$a\Delta_n^b$ & scale $a$ & rate $b$  & description \\
		\hline
		$l_n$   & $\theta$  & $-1/2$    & length of overlapping window for local moving averages \\
		$k_n$   & $\varrho$ & $-\kappa$ & length of disjoint window for spot volatility estimation\\
		$\nu_n$ & $\alpha$  & $\rho$   & truncation level for jumps
	\end{tabular}
\end{table}
With suitable choices of $l_n,\,k_n,\,\nu_n$ in (\ref{def.chat}), this estimator is applicable to any function $g:\mathcal{S}^+_d\mapsto\R^r$ that satisfies
\begin{equation}\label{g.cond}
g\in\mathcal{C}^3(\s)
\end{equation}
where $\s\supset\cup_m\s^\epsilon_m$ for some $\epsilon>0$, $\s^\epsilon_m=\big\{A\in\mathcal{S}^+_d: \inf_{M\in\s_m}\|A-M\|\le\epsilon\big\}$ and $\s_m$ is identified in assumption \ref{A-v}.

\subsection{Choosing tuning parameters}
A proper combination of the tuning parameters is crucial for consistency, CLT, and optimal convergence rate. For these objectives, one needs
\begin{equation}\label{tuning}
\left\{\begin{array}{rcll}
	l_n &\asymp& \theta\Delta_n^{-1/2} &\\
	k_n &\asymp& \varrho\Delta_n^{-\kappa} &\text{ where } \kappa\in\left(\frac{2}{3}\vee\frac{2+\nu}{4},\frac{3}{4}\right)\\
	\nu_n &=& \alpha\Delta_n^{\rho} &\text{ where }\rho\in \left[\frac{1}{4}+\frac{1-\kappa}{2-\nu},\frac{1}{2}\right)
\end{array}\right.
\end{equation}
$\theta,\varrho,\alpha>0$ are positive finite, and $\nu\in[0,1)$ is introduced in assumption \ref{A-v} which dictates the jump intensity. 

The rest of this section offers an intuition for (\ref{tuning}). The reader can skip this part without affecting understanding of the main result in section \ref{sec:asymp}.
\begin{enumerate}
\item \textit{$l_n$ influences the convergence rate}\\
	In the example (\ref{def.Y}), according to (\ref{def.Ubar.Uhat}),
	\[\widebar{Y}^n_i = \widebar{X}^n_i + \widebar{\e}^n_i\]
	and we can write $\widebar{\e}^n_i = -\psi_n^{-1/2}\sum_{h=0}^{l_n-1}(\varphi^n_{h+1}-\varphi^n_h)\e^n_{i+h}$.
	Under the conditional independence of $\e^n_i$'s, $\widebar{\e}^n_i=O_p(l_n^{-1})$; $\widebar{X}^n_i=O_p(\Delta_n^{-1/2})$ by (\ref{est.Xbar}). By taking $l_n\asymp\Delta_n^{-1/2}$ the orders of $\widebar{X}^n_i$ and $\widebar{\e}^n_i$ are equal, this choice of local smoothing window will deliver the optimal rate of convergence.
\item \textit{$k_n$ dictates bias-correction and the CLT form}\\
	Here let's focus on the case $d=1$, $X$ is continuous, then 
	\[\widehat{c}^n_i - c^n_i = d^n_i + s^n_i\]
	\begin{itemize}
	\item $d^n_i=\frac{1}{(k_n-l_n)\Delta_n}\int_{i\Delta_n}^{(i+k_n-l_n+1)\Delta_n}(c_s-c^n_i)\ds s$ is the ``\textit{discretization error}'', $d^n_i=O_p((k_n\Delta_n)^{1/2})$ by (\ref{classic});
	\item $s^n_i\approx\frac{1}{(k_n-l_n)\Delta_n}\Delta_n^{1/4}(\chi^n_{i+k_n-l_n+1}-\chi^n_i)$ is the ``\textit{statistical error}'',  where $\chi$ is a continuous It\^o semimartingale, this result is due to (3.8) in \cite{j09}, so $s^n_i=O_p((k_n\Delta_n^{1/2})^{-1/2})$.
	\end{itemize}
	
	Balancing the orders of $d^n_i$ and $s^n_i$ by setting $\kappa=3/4$ will result in the minimum order of total estimation error. However, in the case $\kappa\ge3/4$ the bias involves volatility of volatility and volatility jump, which are difficult to estimate and de-bias in applications. Therefore, it is advisable to choose $\kappa<3/4$, in which case the statistical error dominates in the bias, thereby the thorny terms are circumvented. Besides, to achieve successful de-biasing of statistical error and negligibility of higher-order Taylor-expansion terms, we need $\kappa>2/3$. 
	
	\begin{tikzpicture}[scale=1.15]
	\draw[->] (0,0) -- (8.0,0) node[anchor=west] {$\kappa$ for $k_n$};
	\draw	(2.0,0)  node[anchor=north] {$1/2$}
	(3.66,0) node[anchor=north] {$2/3$}
	(4.5,0)  node[anchor=north] {$3/4$}
	(7.0,0)  node[anchor=north] {$1$};
	\draw[<->] (3.66,-0.5) -- (4.5,-0.5);
	\fill[red!10]  (7.0,3.0) -- (4.5,1.5) -- (4.5,0) -- (7.0,0) -- cycle;
	\fill[blue!10] (2.0,3.0) -- (4.5,1.5) -- (4.5,0) -- (2.0,0) -- cycle;
	\draw[->] (1,0) -- (1,3.6) 
	node[anchor=south,align=center] {order of $\log(|\widehat{c}^n_i-c^n_i|)$};
	\draw[thick,dotted] (2.0,0)  -- (2.0,3.5);
	\draw[thick,dashed] (3.66,0) -- (3.66,3.5);
	\draw[thick,dashed] (4.5,0)  -- (4.5,3.5);
	\draw[thick,dotted] (7.0,0)  -- (7.0,3.5);
	\draw (7.5,3) node {\textcolor{red}{$d^n_i$}}; 
	\draw[thick,red,dashed] (2,0) -- (7,3);
	\draw[thick,<-] (6.5,1.5) -- (7.2,1.3)
	node[right,align=left] {dominated by\\ discretization\\ error};
	\draw (1.5,3) node {\textcolor{blue}{$s^n_i$}}; 
	\draw[thick,blue] (2,3) -- (7,0);
	\draw[thick,<-] (2.5,1.5) -- (1.8,1.3)
	node[left,align=right] {dominated by\\ statistical error};
	\draw[thick,<-] (4.5,1.6) -- (4.7,2)
	node[above,align=center] {error\\ minimizing\\$\kappa$};
	\end{tikzpicture}
	
	Section 3.1, 3.2 of \cite{jr13} give a similar discussion in absence of noise.
\item \textit{$\nu_n$ disentangles volatility from jump variation}\\
	$\|\widebar{Y}^n_i\|=O_p(\Delta_n^{1/2})$ if there is no jump over $[i\Delta_n,(i+l_n)\Delta_n]$, via (\ref{est.Y.bars.hats}). By choosing $\rho<1/2$, the truncation level, which is $\nu_n>\Delta_n^{1/2}$, keeps the diffusion movements and discards jumps in a certain sense. To effectively filter out the jumps, the truncation level should be bounded above and the upper bounds depends on the jump activity index $\nu$.
\end{enumerate}

\begin{remk}
If the reader is interested to estimate $(\ref{def.S(g)})$ with $g$ satisfying
\begin{equation}\label{g.cond1}
	\|\partial_h g(x)\| \le K_h(1+\|x\|^{r-h}),\,\, h=0,1,2,3,\, r\ge3
\end{equation}
the requirements on $k_n$ and $\nu_n$ can be loosened and become
\begin{equation}\label{tuning1}
\begin{array}{rcl}
	k_n\Delta_n^\kappa &\asymp& \varrho,\, \text{ where } \kappa\in\left(\frac{2}{3},\frac{3}{4}\right)\\
	\nu_n             &=&   \alpha\Delta_n^{\rho},\, \text{ where } \rho\in \left[\frac{1}{4}+\frac{1}{4(2-\nu)},\frac{1}{2}\right)
\end{array}
\end{equation}
For wider applicability, we choose to accommodate the functional space (\ref{g.cond}) and retain the requirement (\ref{tuning}).
\end{remk}

\section{Asymptotics}\label{sec:asymp}
\subsection{Elements}
Before stating the asymptotic result, some elements appearing in the limit need to be defined. Associate the following quantities with the smoothing kernel $\varphi$ for $l,m=0,1$:
\begin{equation}\label{phi.vars}
\begin{array}{ll}
	\phi_0(s)=\int_s^1\varphi(u)\varphi(u-s)\ds u, & \phi_1(s)=\int_s^1\varphi'(u)\varphi'(u-s)\ds u \\
	\Phi_{lm}=\int_0^1\phi_l(s)\phi_m(s)\ds s, & \Psi_{lm}=\int_0^1s\,\phi_l(s)\phi_m(s)\ds s
\end{array}
\end{equation}

Define $\Sigma$, $\Theta$, $\Upsilon$ as $\R^{d\times d\times d\times d}$-valued functions, such that for $x,z\in\R^{d\times d}$, $j,k,l,m=1,\cdots,d$,
\begin{equation}\label{tensors}
\begin{array}{lcl}
	\Sigma(x)^{jk,lm}   &=& x^{jl}x^{km} + x^{jm}x^{kl}\\
	\Theta(x,z)^{jk,lm} &=& x^{jl}z^{km} + x^{jm}z^{kl} + x^{km}z^{jl} + x^{kl}z^{jm}\\
\end{array}
\end{equation}
and $\Xi$ also as a tensor-valued function
\begin{equation}\label{def.Xi}
	\Xi(x,z) = \frac{2\theta}{\phi_0(0)^2}\left[\Phi_{00}\Sigma(x) + \frac{\Phi_{01}}{\theta^2}\Theta(x,z) + \frac{\Phi_{11}}{\theta^4}\Sigma(z) \right]
\end{equation}
where $\theta$ is introduced in (\ref{tuning}).

Now we are ready to describe the limit process.
\begin{defn}\label{def.Z(g)}
Given $g$ satisfying (\ref{g.cond}) or (\ref{g.cond1}), $Z(g)$ is a process defined on an extension of the probability space  $(\Omega,\F,(\F_t),\mathbb{P})$ specified in (\ref{prob.sp}),
such that conditioning on $\mathcal{F}$,  $Z(g)$ is a mean-0 continuous It\^o semimartingale with conditional variance
\begin{equation*}
	\widetilde{E}[Z(g)Z(g)^\T|\mathcal{F}]=V(g)
\end{equation*}
where $\widetilde{E}$ is the conditional expectation operator on the extended probability space and
\begin{equation}\label{AVAR}
	V(g)_t = \int_0^t \sum_{j,k,l,m=1}^{d}\partial_{jk}g(c_s)\,\partial_{lm}g(c_s)^\T\,\Xi(c_s,\gamma_s)^{jk,lm}\ds s
\end{equation}
with $\partial_{jk}g(c)$ being the first-order partial derivative of $g$ with respect to $c^{jk}$, $\rr$ being the variance process of noise defined in (\ref{def.r}).
\end{defn}

\subsection{The formal results}
\begin{thm}\label{CLT}
Assume assumptions \ref{A-v}, \ref{A-r}. Given $g$ satisfying (\ref{g.cond}), we control the tuning parameters $l_n$, $k_n$, $\nu_n$ according to (\ref{tuning}), then we have the following stable convergence in law of discretized process to a conditional continuous It\^o semimartingale on compact subsets of $\R^+$:
\begin{equation}\label{clt1}
	\Delta_n^{-1/4}\left[\widehat{S}(g)^n-S(g)\right]\overset{\mathcal{L}-s(f)}{\longrightarrow}Z(g)
\end{equation}
where $S(g)$ is defined in (\ref{def.S(g)}), $\widehat{S}(g)^n$ is from definition \ref{def.S(g)hat}, $Z(g)$ is identified in definition \ref{def.Z(g)}.
\end{thm}

Theorem \ref{CLT} is valid over the functional space (\ref{g.cond}), which is as general as the current literature can get. If applications require functionals whose derivatives satisfy the polynomial growth condition (\ref{g.cond1}), we can put less restrictions on the tuning parameters.
\begin{thm}\label{CLT.restricted}
Assume assumptions \ref{A-v}, \ref{A-r}. Replace the functional space (\ref{g.cond}) with (\ref{g.cond1}), replace the tuning conditions (\ref{tuning}) on $k_n$, $\nu_n$ with (\ref{tuning1}), then (\ref{clt1}) still holds true.
\end{thm}
However, theorem \ref{CLT.restricted} rules out operations that involve matrix inversion, it is not applicable to, for instance, inference of linear regression models. In the rest of this paper, we focus on the results over the general functional space (\ref{g.cond}).

The asymptotic result is stated with a probabilistic flavor, which is necessary to express the strongest convergence\footnote{It is functional stable convergence (or stable convergence of processes) in law.} by appendix \ref{apdx:derivation}. There is an alternative formulation which is more relevant for statistical applications:
\begin{equation}\label{clt2}
	n^{1/4}\left[\widehat{S}(g)^n_t-S(g)_t\right]\overset{\mathcal{L}-s}{\longrightarrow}\mathcal{MN}\big(0,\sqrt{t}V(g)_t\big)
\end{equation}
this is true under the same conditions and $t$ is finite.

\subsection{Confidence intervals}
The asymptotic variance in (\ref{clt2}) can be estimated by plugging in spot estimates of volatility (\ref{def.chat}) and noise covariance matrix (\ref{def.rhat}):
\begin{equation}\label{def.V(g)hat}
	\widehat{V}(g)^n_t \equiv k_n\Delta_n\sum_{i=0}^{N^n_t-1}\sum_{j,k,l,m=1}^{d}\partial_{jk}g(\widehat{c}^n_{ik_n})\,\partial_{lm}g(\widehat{c}^n_{ik_n})^\T\,\Xi(\widehat{c}^n_{ik_n},\widehat{\gamma}^n_{ik_n})^{jk,lm}
\end{equation}

\begin{prop}\label{AVAR.est}
$\widehat{V}(g)^n_t$ is consistent under (\ref{tuning}) and assumptions \ref{A-v}, \ref{A-r}. Specifically, for all finite $t$,
\begin{equation*}
	\big\|\widehat{V}(g)^n_t - V(g)_t\big\| = O_p(\Delta_n^{\kappa-1/2})
\end{equation*}
where $\kappa$ is specified in (\ref{tuning}).
\end{prop}
\begin{proof}
The asymptotic variance (\ref{AVAR}) is a smooth functional of spot volatility and instantaneous noise covariance, so the consistency of (\ref{def.V(g)hat}) follows from the consistence of the spot volatility estimator (\ref{def.chat}) and the noise covariance estimator (\ref{def.rhat}). According to lemma \ref{est.beta} and (\ref{est.chi}), the error rate of (\ref{def.V(g)hat}) is determined by the estimation error of spot volatility. Therefore, the error rate of (\ref{def.V(g)hat}) is the same as the error rate of the volatility functional estimator without bias correction, which is $(k_n\Delta_n^{1/2})^{-1}$, then the proposition follows from (\ref{tuning}).
\end{proof}

Based on theorem \ref{CLT}, proposition \ref{AVAR.est} and the property of stable convergence, we have the following feasible central limit theorem:
\begin{corol}
Under (\ref{tuning}) and assumptions \ref{A-v}, \ref{A-r}, we have
\begin{equation}\label{clt3}
	\big[\Delta_n^{1/2}\,\widehat{V}(g)^n_t\big]^{-1/2}\left[\widehat{S}(g)^n_t-S(g)_t\right]\overset{\mathcal{L}}{\longrightarrow}\mathcal{N}\big(0,\mathbb{I}\big)
\end{equation}
in restriction to the event $\{\omega\in\Omega, \widehat{V}(g)^n_t \text{ is positive definite}\}$, where $\Omega$ is defined in (\ref{prob.sp}).
\end{corol}

\section{Applications}\label{sec:appl}
\subsection{Quarticity estimation}
In the univariate setting, the so-called quarticity $\int_0^tc_s^2\ds s$ appears in the asymptotic variances of many extant volatility estimators. The multivariate counterpart involves $\int_0^t c_s^{jl}c_s^{km}+c_s^{jm}c_s^{kl}\ds s$, e.g., $\Xi(c_s,\rr_s)^{jk,lm}$ in (\ref{AVAR}). Since the quarticity is an integrated functional of volatility, the volatility functional estimator facilitates uncertainty quantification for various volatility estimators.

\subsection{Realized Laplace transform}
\cite{tt12} put forward an estimator of the realized Laplace transform of volatility defined as
	\[\int_0^t e^{iwc_s}\ds s.\]
This transform can be viewed as the characteristic function of volatility under the occupation measure. By matching the the moments of realized Laplace transform with those induced by a model, we can estimate model parameter(s) or test the model. An open question noted by \cite{tt12} is the estimation of realized Laplace transform using noisy data. By the nonparametric estimation of volatility path in the first stage and the bias-corrected Riemann summation of functional plug-ins in the second stage, this paper contributes a rate-optimal solution to the open question.

\subsection{Generalized method of moments (GMM)}
\cite{lx16} proposed the generalized method of integrated moments for financial high-frequency data. In estimating an option pricing model, one observes the process $Z_t = (t, X_t, r_t, d_t)$ where $X_t$ is the price of the underlying observed without any noise, $r_t$ is the short-term interest rate, $d_t$ is the dividend yield. One model of the arbitrage-free option price under the risk-neutral probability measure is
\begin{equation*}
	\beta_t = f(Z_t,c_t;\theta^*)
\end{equation*}
where $f$ is deterministic, $\theta^*$ is the true model parameter. The observed option price is often modeled as
\begin{equation*}
	Y_{i\Delta_n} = \beta_{i\Delta_n} + \epsilon_i
\end{equation*}
where $\epsilon_i$ is pricing error and $\E(\epsilon_i)=0$. Let $g(Z_t,c_t;\theta) = \E[Y_t - f(Z_t,c_t;\theta)]$, then we have the following integrated moment condition:
\begin{equation*}
	G(\theta^*) = 0
\end{equation*}
where $G(\theta) = \int_0^t g(Z_s,c_s;\theta)\,\mathrm{d}s$. Utilizing noisy observations of $X$ at higher frequencies, $\widehat{S}(g)_t^n$ of this paper provides a means to compute a bias-corrected sample moment function of GMM.

\subsection{Linear regression}
In the practice of linear factor models and financial hedging, one faces the tasks of computing the factor loadings and the hedge ratios. These tasks can be formulated as the estimation of the coefficient $\beta$ in the time-series linear regression model
\begin{equation*}
Z_t^c = \beta^\T S_t^c + R_t
\end{equation*}
where
\begin{equation*}
\left\{\begin{array}{cl}
S_t\equiv& S_0 + \int_0^tb^S_u\ds u + \int_0^t\sigma^S_u\ds W^S_u + J^S_t\\
Z_t\equiv& Z_0 + \int_0^tb^Z_u\ds u + \beta^\T\int_0^t\sigma^S_u\ds W^S_u + \int_0^t\sigma^R_u\ds W^R_u + J^Z_t
\end{array}\right.
\end{equation*}
$\langle W^S,W^R\rangle=0$, $S_t\in\R^{d-1}$, $Z_t\in\R$, and $S^c$, $Z^c$ are the continuous parts of the It\^o semimartingales.

Let $X = (S^\T,Z)^\T$, we can write $X_t=X_0 + \int_0^tb_u\ds u + \int_0^t\sigma_u\ds W_u + J_t$ where $b=(b^{S,T}, b^Z)^\T$, $W=(W^{S,\T},W^R)^\T$, $J=(J^{S,\T},J^Z)^\T$ and
\begin{equation*}
\sigma = \left[
\begin{array}{cc}
\sigma^S & 0 \\
\beta^\T\sigma^S & \sigma^R
\end{array}\right]
\end{equation*}
so
\begin{equation*}
c = \sigma\sigma^\T = 
\left[\begin{array}{cc}
\sigma^S\sigma^{S,\T}         & \sigma^S\sigma^{S,\T}\beta \\
\beta^\T\sigma^S\sigma^{S,\T} & \beta^\T\sigma^S\sigma^{S,\T}\beta + (\sigma^R)^2
\end{array}\right]
\equiv
\left[\begin{array}{cc}
c^{SS} & c^{SZ} \\
c^{ZS} & c^{ZZ}
\end{array}\right]
\end{equation*}
hence by letting $g(c)=s^{SS,-1}c^{SZ}$, we have $\beta=t^{-1}S(g)_t$. \cite{ltt17} proposed this method for the situation in which the process $X$ can be perfectly observed. When the observations contain noise, the methodology of this paper can extend the estimator of \cite{ltt17} to wider applicability.

\subsection{Principal component analysis (PCA)}
An interesting question about stochastic volatility is its spectral structure $c_sv_s = \lambda_sv_s$.
\cite{ax19} applied PCA to nonstationary financial data by conducting inference on the realized eigenvalue $\int_0^t\lambda_s\ds s$, realized eigenvector $\int_0^t v_s\ds s$, realized principal component $\int_0^tv_{s-}\ds X_s$. In the basic setting where $\lambda_s$ is a simple eigenvalue of $c_s$ and $v_s$ is the corresponding eigenvector, $g(c_s)=\lambda_s$ and $g(c_s) = v_s$ are three-times continuously differentiable, therefore the inferential results of $S(g)$ are applicable. More recently, \cite{cmz18} extends the realized PCA to asynchronously observed high-dimensional noisy data,
while this paper extends the realized PCA methodology to be both noise-robust and rate-optimal.

\section{Simulation}
As a proof of concept, estimators corresponding to $g(c)=c^2$, $g(c)=\log(c)$ when $d=1$ are calculated based on the simulation model
\begin{equation*}
\left\{\begin{array}{lcl}
Y^n_i        &=&X^n_i + \e^n_i \\
\mathrm{d}X_t&=&.03\ds t + \sqrt{c_t}\ds W_t + J^X_t\ds N^X_t\\
\mathrm{d}c_t&=&6(.16-c_t)\ds t + .5\sqrt{c_t}\ds B_t + \sqrt{c_{t-}}J^c_t\ds N^c_t
\end{array}\right.
\end{equation*}
where $\e^n_i\overset{\text{i.i.d.}}{\sim}N(0,.005^2)$, $\E[(W_{t+\Delta}-W_t)(B_{t+\Delta}-B_t)]=-.6\Delta$, $J^X_t\sim N(-.01,.02^2)$, $N^X_{t+\Delta}-N^X_t\sim\mathrm{Poisson}(36\Delta)$, $\log(J^c_t)\sim N(-5,.8)$, $N^c_{t+\Delta}-N^c_t\sim\mathrm{Poisson}(12\Delta)$. 

Each simulation employs $23400\times21$ data points with $\Delta_n=1s$. We choose the following tuning parameters:
\begin{table}[H]
\begin{tabular}{l|lll}
	functionals    & $l_n$ & $k_n$ & $\nu_n$ \\
	\hline
	$g(c)=c^2$     & $\lfloor\Delta_n^{-.5}\rfloor$ & $\lfloor\Delta_n^{-.69}\rfloor$ & $1.6\overline{\sigma}^2\Delta_n^{.47}$ \\
	$g(c)=\log(c)$ & $\lfloor\Delta_n^{-.5}\rfloor$ & $\lfloor\Delta_n^{-.7}\rfloor$  & $1.5\overline{\sigma}^2\Delta_n^{.47}$
\end{tabular}
\end{table}
where $\overline{\sigma}^2$ is an estimate of the average volatility by bipower variation \cite{pv09a}.

The results are shown in figure \ref{MC}.
\begin{figure}[H]
	\centering
	\caption{Simulation of volatility functional estimators}\label{MC}
	\includegraphics[width=.5\textwidth]{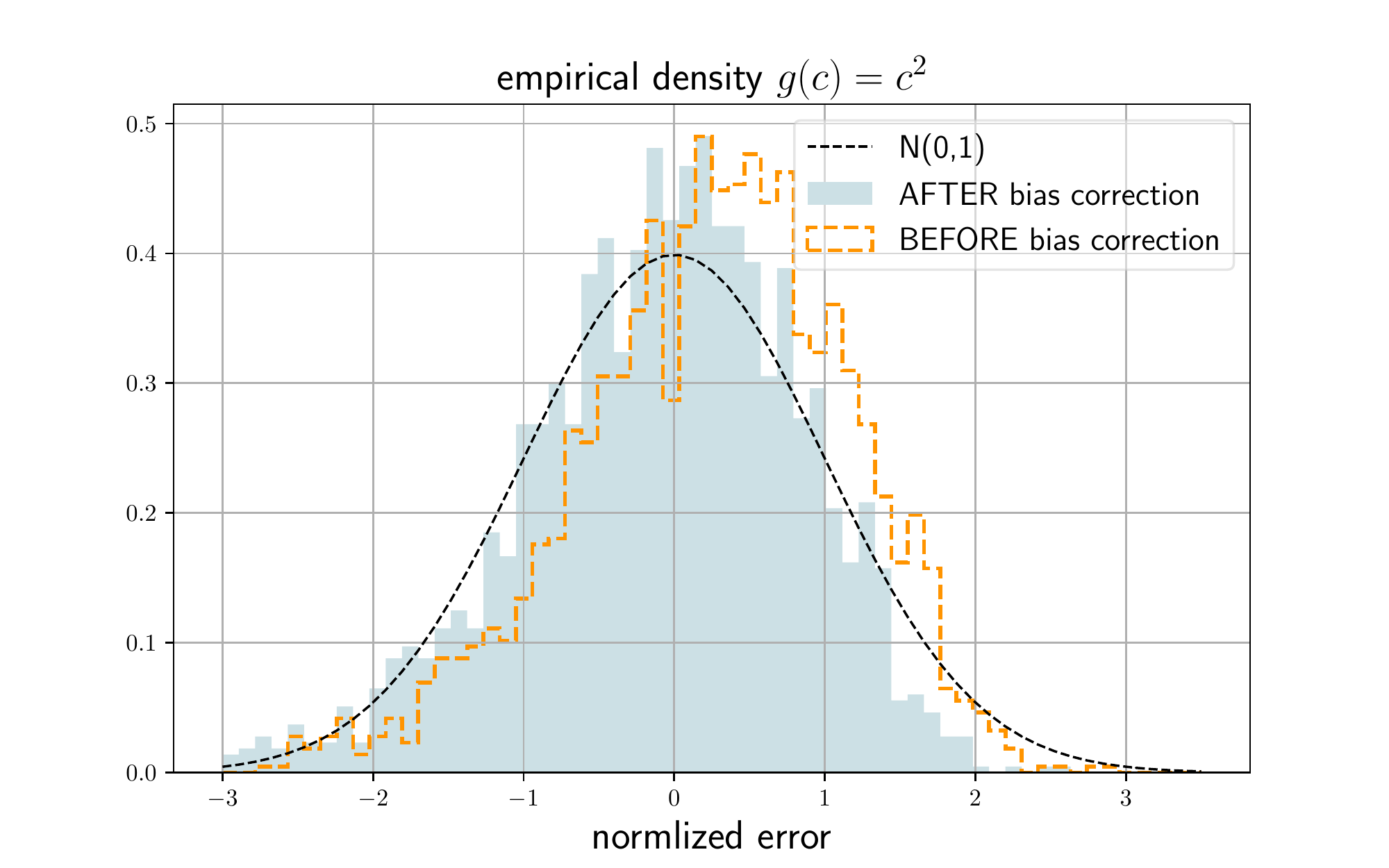}\includegraphics[width=.5\textwidth]{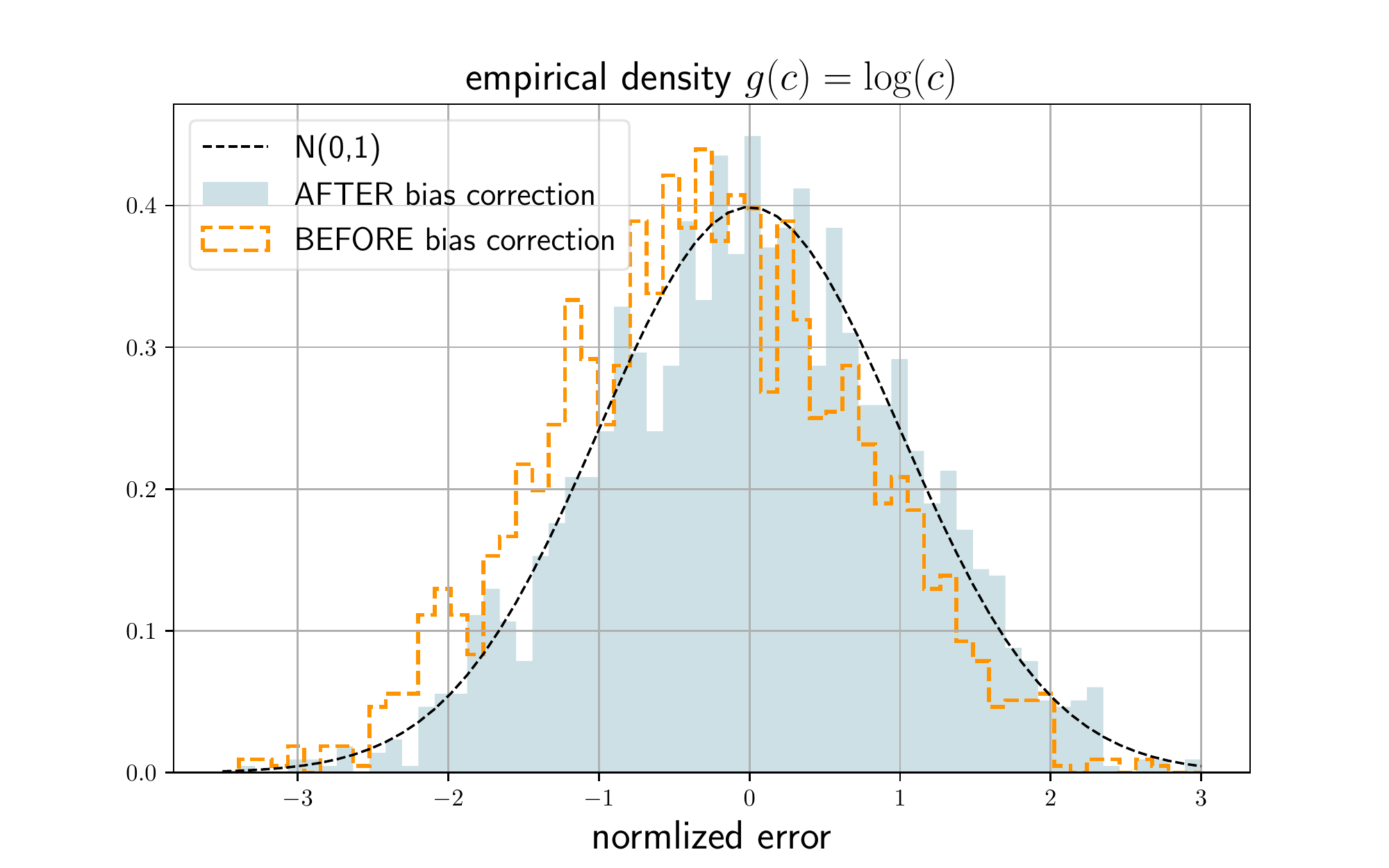}
\end{figure}

\section{Discussions}
\subsection{Jump truncation}
It is worthwhile to take account of the dimension $d$ and volatility levels in jump truncation. One possibility is to use the truncation indicator $\prod_{r=1}^d\1{|\widebar{Y}^{r,n}_i|\le\alpha_r\Delta_n^\rho}$ where $\alpha_r$ is related to the volatility of the $r$-th component.

\subsection{Semi-efficiency}

The efficient bound of volatility estimation is studied by \cite{r11}. Nonetheless, pre-averaging hasn't been able to attain the efficiency bound. According to \cite{jm15}, by taking the moving average adaptively in the time domain, the asymptotic variance of the pre-averaging method can be within 7\% of the efficiency bound. Using adaptive pre-averaging in volatility functional estimation is beyond the scope of this paper and is currently under investigation.

An efficient alternative is the spectral approach \cite{bhmr14, ab15}. The multi-scale approach \cite{z06}, realized kernels \cite{b08}, quasi-likelihood \cite{x10} are equally capable to rate-optimally handle noise. In the univariate case realized kernels and quasi-likelihood could be improved to be nearly efficient, cf. \cite{cp18}. The pre-averaging are adopted in this paper to simultaneously handle price jumps and microstructure noise.

\subsection{Finite-sample consideration}
First, we consider effective jump truncations. It is worthwhile to take account of the dimension $d$ and volatility levels in jump truncation. One possibility is to use the truncation indicator $\prod_{r=1}^d\1{|\widebar{Y}^{r,n}_i|\le\alpha_r\Delta_n^\rho}$ where $\alpha_r$ is related to the volatility of the $r$-th component.

Next, in finite sample, the spot volatility estimator (\ref{def.chat}) might not be positive semidefinite due to the noise-correction term $\widehat{Y}^n_i$. \cite{ckp10} suggests to increase $l_n$ to attenuate noise in $\widebar{Y}^n_i$ and dispense with $\widehat{Y}^n_i$:
\begin{equation*}
	\widetilde{c}^n_i \equiv \frac{1}{(k_n-l_n)\Delta_n}\sum_{h=1}^{k_n-l_n+1}\widebar{Y}^n_{i+h}\cdot\widebar{Y}^{n,\mathrm{T}}_{i+h}\1{\|\widebar{Y}^n_{i+h}\|\le\nu_n}
\end{equation*}
and let $l_n\asymp\theta\Delta_n^{-1/2-\delta}$ where $\delta\in(.1,.5)$. 
According to \cite{c19b}, if one plugs in $\widetilde{c}^n_i$ with the tunning parameters satisfying
\begin{equation*}
\left\{\begin{array}{rcll}
	k_n &\asymp& \varrho\Delta_n^{-\kappa} & \kappa\in\left(\big(\frac{2}{3}+\frac{2\delta}{3}\big)\vee\big(\frac{2+\nu}{4}+\frac{(2-\nu)\delta}{2}\big),\frac{3}{4}+\frac{\delta}{2}\right)\\
	\nu_n &=& \alpha\Delta_n^{\rho} &\rho\in \left[\frac{1}{4}+\frac{\delta}{2}+\frac{1-\kappa}{2-\nu},\frac{1}{2}\right)
\end{array}\right.
\end{equation*}
another central limit theorem holds after a different bias correction. However, doing so sacrifices the convergence rate, which drops from $n^{1/4}$ down to $n^{1/4-\delta/2}$ (strictly less than $n^{1/5}$). Moreover, the choice of tunning parameters becomes less robust compared to (\ref{tuning}).

To preserve the optimal convergence rate in the unfortunate event where the spot volatility estimator is not positive semidefinite, it is advisable to project $\widehat{c}^n_i$ onto the convex cone $\mathcal{S}^+_d$ with respect to the Frobenius norm. Suppose $\widehat{c}^n_i = Q\Lambda Q^\T$ is the eigenvalue factorization, the positive semidefinite projection is $\widehat{c}'^n_i=Q\Lambda_+Q^\T$ where $\Lambda_+^{jj} = \Lambda^{jj}\vee 0$. By the convex geometry of $\mathcal{S}^+_d$, we have $\|\widehat{c}'^n_i-c^n_i\|\le \|\widehat{c}^n_i-c^n_i\|$, hence the convergence rate is retained.



\appendix
\section{Assumptions}\label{sec:assump}
This section presents details of model specification and assumptions. 

The pure jump process is
\begin{equation}\label{J}
	J_t = \int_{(0,t]\times E}\delta(s,x)\,\mathfrak{p}(\mathrm{d}s,\ds x)
\end{equation}
where $\delta$ is a $\mathbb{R}^d$-valued predictable function on $\R^+\times E$, $E$ is a Polish space, $\mathfrak{p}$ is a Poisson random measure with compensator $\mathfrak{q}(\ds u, \ds x)=\ds u\otimes\lambda(\ds x)$, $\lambda$ is a $\sigma$-finite measure on $E$ and has no atom. The volatility process is assumed to be an It\^o semimartingale
\begin{equation}\label{c}
	c_t=c_0+\int_0^tb^{(c)}_s\ds s+\int_0^t\sigma^{(c)}_s\ds W_s + \int_{(0,t]\times E}\delta^{(c)}(s,x)\,(\mathfrak{p}-\mathfrak{q})(\mathrm{d}s,\ds x)
\end{equation}
where $b^{(c)}$ is $\R^{d\times d}$-valued, optional, c\`adl\`ag; $\sigma^{(c)}$ is $\R^{d \times d \times d'}$-valued, adapted, c\`adl\`ag; $\delta^{(c)}$ is a $\mathbb{R}^{d\times d}$-valued predictable function on $\R^+\times E$.\footnote{To guarantee positive semidefiniteness in concrete applications, one needs to impose parametric restrictions on the nonparametric model (\ref{c}). The result in this paper holds true for any process $c$ which is positive semidefinite and satisfies (\ref{c}).}

Let $\big(\Omega^{(0)},\F^{(0)},(\F^{(0)}_t),\mathbb{P}^{(0)}\big)$ be a filtered probability space with respect to which $X$, $c$ are $(\F^{(0)}_t)$-adapted; let $\big(\Omega^{(1)},\F^{(1)},(\F^{(1)}_t),\mathbb{P}^{(1)}\big)$ be another filtered probability space accommodating $Y$; $\forall t\ge0$, $\forall A\in\F^{(0)}$, let $Q_t(A,\cdot)$ be a conditional probability measure on $\left(\Omega^{(1)},\F^{(1)}\right)$. The conditional noise variance process is defined as
\begin{equation}\label{def.r}
	\gamma_t = \int_{\Omega^{(1)}}Y_t(\omega)Y_t(\omega)^\T\,Q_t(\cdot,\mathrm{d}\omega)-X_tX_t^\T 
\end{equation}
All the stochastic dynamics above can be described on the filtered extension $\left(\Omega,\F,(\mathcal{F}_t),\mathbb{P}\right)$, where
\begin{equation}\label{prob.sp}
\left\{\begin{array}{l}
	\Omega=\Omega^{(0)}\times\Omega^{(1)}\\
	\F=\F^{(0)}\otimes\F^{(1)}\\
	\F_t=\bigcap_{s>t}\big(\F^{(0)}_s\otimes\F_s^{(1)}\big),\hspace{5mm}  \widetilde{\F}_t=\bigcap_{s>t}\big(\F^{(0)}\otimes\F_s^{(1)}\big)\\
	\mathbb{P}\left(A\times \mathrm{d}\omega\right)=\mathbb{P}^{(0)}(A)\cdot\otimes_{t\ge0}Q_t(A,\ds\omega),\,\forall A\in\F^{(0)}
\end{array}\right.
\end{equation}
In the sequel, $\E(\cdot)$ denotes the expectation operator on $(\Omega^{(0)},\F^{(0)})$ or $(\Omega,\F)$; $E(\cdot|\mathcal{H})$ denotes the conditional expectation operator, with $\mathcal{H}$ being $\F^{(0)}_t$, $\F^{(1)}_t$, $\F_t$, $\widetilde{\F}_t$; $E^n_i(\cdot)$ denotes $E(\cdot|\F^n_i)$. Assumptions are collected below.
\begin{assump}{A-$\nu$}[regularity]\label{A-v} 
$b$ has $\frac{1}{2}$-H\"older sample path, i.e., $\forall t,s\ge0$,
\begin{equation*}
	E\Big(\sup_{u\in[0,s]}\|b_{t+u}-b_t\||\F^{(0)}_t\Big)\le Ks^{1/2},\,a.s.;
\end{equation*}
$c$ is of the form (\ref{c}), there is a sequence of triples $(\tau_m,\mathcal{S}_m,\Gamma_m)$, where $\tau_m$ is a stopping time and $\tau_m\nearrow\infty$;  $\s_m\subset\mathcal{S}^+_d$ is convex, compact such that 
\begin{equation*}
	t\in[0,\tau_m] \Rightarrow c_t\in\s_m;
\end{equation*} 
$\Gamma_m$ is a sequence of bounded $\lambda$-integrable functions on $E$, such that
\begin{equation*}
	t\in[0,\tau_m] \Longrightarrow \left\{
	\begin{array}{l}
		\|b_t\| + \|\sigma_t\| + \|b^{(c)}_t\| + \|\sigma^{(c)}_t\| \le m\\
		\|\delta(t,x)\|^\nu\wedge1\le\Gamma_m(x),\,\,\nu\in[0,1)\\
		\|\delta^{(c)}(t,x)\|^2\wedge1\le\Gamma_m(x)
	\end{array}\right.
\end{equation*}
\end{assump}

\begin{assump}{A-$\gamma$}[noise]\label{A-r}
$\forall t\in\mathbb{R}^+$,
	\[\int_{\Omega^{(1)}}Y_t(\omega)\,Q_t(\cdot,\ds\omega)=X_t\]
$\forall t\ne s$, $\forall A\in\F^{(0)}_{s\wedge t}$
\[\int_{\Omega^{(1)}\times\Omega^{(1)}}(Y_t(\omega)-X_t)(Y_s(\omega)-X_s)^\T\,Q_t(A,\ds\omega)\,Q_s(A,\ds\omega)=0\]
furthermore,
\begin{equation*}
	\gamma_t = \gamma_0+\int_0^tb^{(r)}_s\ds s + \int_0^t\sigma^{(r)}_s\ds W_s + \int_{(0,t]\times E}\delta^{(r)}(s,x)\,\mathfrak{p}(\mathrm{d}s,\ds x)
\end{equation*}
for the same $\tau_m$, $\Gamma_m$ in assumption \ref{A-v},
\begin{equation*}
	t\in[0,\tau_m] \Longrightarrow \left\{
	\begin{array}{l}
		\|b^{(r)}_t\| + \|\sigma^{(r)}_t\|\le m\\
		\|\delta^{(r)}(t,x)\|^2\wedge1\le\Gamma_m(x)
	\end{array}\right.
\end{equation*}
\end{assump}

\section{Derivation}\label{apdx:derivation}
\subsection{Preliminaries}
In the sequel, the constant $K$ changes across lines but remains finite; $K_q$ is a constant depending on $q$; $a_n=O_p(b_n)$ means $\forall\epsilon>0$, $\exists M>0$ such that $\sup_n\mathbb{P}(a_n/b_n>M)<\epsilon$ ; $a_n\asymp b_n$ means both $a_n/b_n$ and $b_n/a_n$ are bounded for large $n$. Six useful results are stated below.

I. By a localization argument from section 4.4.1 in \cite{jp12}, without loss of generality we can assume $\exists$ a constant K, a bounded $\lambda$-integrable function $\Gamma$ on $E$, a convex compact subset $\mathcal{S}\in\mathcal{S}^+_d$ and $\epsilon>0$, $g\in\mathcal{C}^3(\mathcal{S}^\epsilon)$ where $\mathcal{S}^\epsilon$ denotes the $\epsilon$-enlargement of $\mathcal{S}$ (see (\ref{g.cond})), such that
\begin{equation}\label{SA-v}
\left\{\begin{array}{l}
	\|b\| + \|\sigma\| + \|b^{(c)}\| + \|\sigma^{(c)}\| \le K\\
	\|\delta(t,x)\|^\nu\wedge1\le\Gamma(x),\,\,\nu\in[0,1)\\
	\|\delta^{(c)}(t,x)\|^2\wedge1\le\Gamma(x)\\
	c\in\mathcal{S}
\end{array}\right.
\end{equation}

II. Define a continuous It\^o semimartingale with parameters corresponding to those in (\ref{def.X}),
\begin{equation*}
X'_t = X_0 + \int_0^tb_s\ds s + \int_0^t\sigma_s\ds W_s
\end{equation*}
Let $Y^*= Y - X + X'$. Based on (\ref{def.Ubar.Uhat}), define
\begin{equation*}
\widehat{c}^{*n}_{i} = \frac{1}{(k_n-l_n)\Delta_n}\sum_{h=1}^{k_n-l_n+1}\left(\widebar{Y}^{*n}_{i+h}\cdot\widebar{Y}^{*n,\T}_{i+h} - \widehat{Y}^{*n}_{i+h}\right)
\end{equation*}
In the upcoming derivation, $\|\widehat{c}^n_i-\widehat{c}^{*n}_i\|$ is tightly bounded provided $\nu_n$ is properly chosen, the focus then will be shifted from $\widehat{c}^n_i$ to $\widehat{c}^{*n}_i$.

III. By estimates of It\^o semimartingale increments, $\forall$ finite stopping time $\tau$
\begin{equation}\label{classic}
	\left\{\begin{array}{ll}
		\big\|E\big(X'_{\tau+s}-X'_{\tau}|\F^{(0)}_\tau\big)\big\| + \big\|E\big(c_{\tau+s}-c_{\tau}|\F^{(0)}_\tau\big)\big\| \\
		\hspace{39mm}+ \big\|E\big(\gamma_{\tau+s}-\gamma_\tau|\F^{(0)}_\tau\big)\big\|& \le Ks\\
		E\left(\sup_{u\in[0,s]}\left\|X'_{\tau+u}-X'_\tau\right\|^q|\F^{(0)}_\tau\right)& \le Ks^{q/2}\\
		E\left(\sup_{u\in[0,s]}\left\|c_{\tau+u}-c_\tau\right\|^q + \left\|\gamma_{\tau+u}-\gamma_\tau\right\|^q|\F^{(0)}_\tau\right)& \le Ks^{(q/2)\wedge1}
\end{array}\right.
\end{equation}
by Lemma 2.1.7, Corollary 2.1.9 in \cite{jp12}
\begin{equation}\label{jump.bounds}
\left\{\begin{array}{l}
	E\left(\sup_{u\in[0,s]}\|J_{\tau+u}-J_\tau\|^q|\F^{(0)}_\tau\right)\le K_q\,sE\big[\widehat{\delta}(q)_{\tau,s}|\F^{(0)}_\tau\big]\\
	E\left[\sup_{u\in[0,s\wedge1]}\left(\frac{\|J_{\tau+u}-J_\tau\|}{s^w}\wedge1\right)^q|\F^{(0)}_\tau\right]\le
	K\,s^{(1-w\nu)(q/\nu\wedge1)}a(s)
\end{array}\right.
\end{equation}
where $\widehat{\delta}(q)_{t,s}\equiv s^{-1}\int_t^{t+s}\int_E\|\delta(u,x)\|^q\,\lambda(\mathrm{d}x)\ds u$ and $a(s)\to0$ as $s\to0$.

IV. Let $\varphi_n(t)=\sum_{h=1}^{l_n-1}\varphi^n_h\mathds{1}_{((h-1)\Delta_n,h\Delta_n]}(t)$. For a generic process $U$, define
\begin{equation}\label{def.Un}
	U^n_{t,s} = \int_t^{t+s}\varphi_n(u-t)\ds U_u
\end{equation}
this quantity is useful in analyzing $\widebar{U}^n_i$.

V. For $p\in\mathbb{N}^+$, $l,m=0,1$, by (\ref{phi.vars}) and Riemann summation,
\begin{equation}\label{sum.phi.quadratic}
\sum_{h=i}^{i+pl_n-2}\,\sum_{h'=h+1}^{i+pl_n-1}\phi_l\Big(\frac{h'-h}{l_n}\Big)\phi_m\Big(\frac{h'-h}{l_n}\Big) = l_n^2\left(p\Phi_{lm}-\Psi_{lm}\right) + O(pl_n)
\end{equation}

VI. By Jensen's inequality and Doob's maximal inequality, we have the following lemma:
\begin{lem}\label{Doob.max.ineq}
	Let $Z_i,i=1,\cdots,M$ be random variables, $\mathcal{H}_i=\sigma(Z_1,\cdots,Z_i)$ be the $\sigma$-algebra generated by $Z_1,\cdots,Z_i$, then
	\begin{equation*}
	\E\left(\sup_{m=1,\cdots,M}\bigg\|\sum_{i=1}^m\left[Z_i - E\left(Z_i|\mathcal{H}_i\right)\right]\bigg\|\right)
	\le K\left(\sum_{i=1}^{M}\E\left(\|Z_i\|^2\right)\right)^{1/2}
	\end{equation*}
\end{lem}

\subsection{Properties of spot estimator: I. jumps}
By assumption \ref{A-v}, \ref{A-r}, (\ref{tuning}), (\ref{SA-v}), (\ref{classic}), (\ref{jump.bounds}), and $\widebar{J}^n_i = \psi_n^{-1/2}\widebar{J}^n_{i\Delta_n,(l_n-1)\Delta_n}$ from (\ref{def.Un}),
\begin{equation}\label{est.Y.bars.hats}
\left\{\begin{array}{lcl}
	E^n_i\left(\big\|\widebar{Y}^{*n}_i\big\|^q\right) &\le& K_q\Delta_n^{q/2}\\
	E^n_i\left(\big\|\widebar{Y}^n_i\big\|^q|\right) &\le& K_q\Delta_n^{(q/2)\wedge(q/4+1/2)}\\
	E^n_i\left(\big\|\widehat{Y}^{*n}_i\big\|^q \vee \big\|\widehat{Y}^n_i\big\|^q\right) &\le& K_q\Delta_n^q\\
	E^n_i\left[\left(\frac{\|\widebar{J}^n_i\|}{\Delta_n^w}\wedge1\right)^q\right] &\le& K_q\Delta_n^{[1/2-(w-1/4)\nu]\times[1\wedge(q/\nu)]}\,a_n
\end{array}\right.
\end{equation}
for some $a_n\to0$. We can write
	\[\left\|\big(\widebar{Y}^n_i\cdot\widebar{Y}^{n,\mathrm{T}}_i\1{\|\widebar{Y}^n_i\|\le\nu_n} - \widehat{Y}^n_i\big) - \big(\widebar{Y}^{*n}_i\cdot\widebar{Y}^{*n,\mathrm{T}}_i - \widehat{Y}^{*n}_i\big)\right\| \le \sum_{r=1}^3\eta^{n,r}_i\]
where
\begin{eqnarray*}
	\eta^{n,1}_i &=& \big\| \widebar{Y}^n_i\cdot\widebar{Y}^{n,\mathrm{T}}_i\1{\|\widebar{Y}^n_i\|\le\nu_n} - \widebar{Y}^{*n}_i\cdot\widebar{Y}^{*n,\mathrm{T}}_i\1{\|\widebar{Y}^{*n}_i\|\le\nu_n} \big\|\\
	\eta^{n,2}_i &=& \big\|\widehat{Y}^n_i - \widehat{Y}^{*n}_i\big\|\\
	\eta^{n,3}_i &=& \big\|\widebar{Y}^{*n}_i\big\|^2 \1{\|\widebar{Y}^{*n}_i\|>\nu_n}
\end{eqnarray*}

Let $u_n=\nu_n/\Delta_n^{1/2}$, $Z^n_i=\|\widebar{Y}^{*n}_i\|/\Delta_n^{1/2}$, $Q^n_i=(\|\widebar{J}^n_i\|/\Delta_n^{1/2})\wedge1$, $V^n_i=(\|\widebar{J}^n_i\|/\Delta_n^{\rho})\wedge1$, we have
	\[\Delta_n\eta^{n,1}_i \le u_n^{-2/(1-2\rho)}(Z^n_i)^{2+2/(1-2\rho)} + (1+Z^n_i) \left[Q^n_i + u_n^2 (V^n_i)^2\right]\]
By successive conditioning and (\ref{jump.bounds}), there is a sequence $a_n\to0$ such that
\begin{equation*}
	E^n_i\big[(\eta^{n,1}_i)^q\big] \le K_q\Delta_n^{2\rho q + 1/2-(\rho-1/4)\nu}a_n
\end{equation*}
Analyzing $\eta^{n,2}_i$ with (\ref{classic}), (\ref{jump.bounds}), analyzing $\eta^{n,3}_i$ with Cauchy-Schwarz inequality, Markov's inequality, (\ref{est.Y.bars.hats}), we can get the following lemma:
\begin{lem}\label{est.chat-chat*}
	Assume (\ref{tuning}), (\ref{SA-v}), assumption \ref{A-v}, \ref{A-r}, then $\exists\,a_n\to0$ such that
	\begin{equation*}
		E^n_i\left(\|\widehat{c}^n_i-\widehat{c}^{*n}_i\|^q\right) \le K_q\left(a_n\Delta_n^{1/2-(\rho-1/4)\nu-(1-2\rho)q} + \Delta_n^{1/2} \right)
	\end{equation*}
\end{lem}

\subsection{Properties of spot estimator: II. continuous part}
\subsubsection{variables}
``If there is a rifle handing on the wall in act one, it must be fired in the next act. Otherwise it has no business being there'', said the Russian playwright Anton Chekhov. Define
\begin{eqnarray*}
	\widebar{C}^n_i&=&\frac{1}{\psi_n}\sum_{h=1}^{l_n-1}(\varphi^n_h)^2\Delta^n_{i+h}C,\,\, C_t=\int_0^tc_s\ds s\\
	D^n_i&=&\widebar{C}^n_i-c^n_i\Delta_n\\
	\Gamma^n_h&=&\Gamma^n_{h,h},\,\,\Gamma^n_{h,h'}=\frac{1}{\psi_n}\sum_{v=h\vee h'}^{h\wedge h'+l_n-1}(\varphi^n_{v-h+1}-\varphi^n_{n-h})(\varphi^n_{v-h'+1}-\varphi^n_{v-h'})\gamma^n_v\\
	R^n_i&=&\widehat{Y}^{*n}_i - \Gamma^n_i\\
	\zeta^n_i&=&\widebar{Y}^{*n}_i\cdot\widebar{Y}^{*n,\T}_i-\widebar{C}^n_i-\Gamma^n_i
\end{eqnarray*}
given $p\in \mathbb{N}^+$, define
\begin{eqnarray*}
	\zeta(W,p)^n_i &=& \sum_{h=i}^{i+pl_n-1}\big[(\sigma^n_i\widebar{W}^n_h)\cdot(\sigma^n_i\widebar{W}^n_h)^\T-\widebar{C}^n_h\big]\\
	\zeta(X,p)^n_i &=& \sum_{h=i}^{i+pl_n-1}\big(\widebar{X}^n_h\cdot\widebar{X}^{n,\T}_h - \widebar{C}^n_h\big)\\
	\zeta(X,p)'^n_i &=& \sum_{h=i}^{i+pl_n-2}\,\sum_{h'=h+1}^{i+pl_n-1}\widebar{X}^n_h\cdot\widebar{X}^{n,\T}_{h'}\phi_1\Big(\frac{h'-h}{l_n}\Big)\\
	\zeta(p)^n_i &=& \sum_{h=i}^{i+pl_n-1}\zeta^n_h
\end{eqnarray*}
let $m(n,p)=\left\lfloor\frac{k_n}{(p+1)l_n}\right\rfloor$, $a(n,p,h)=1+h(p+1)l_n$, $b(n,p,h)=a(n,p,h)+pl_n$, then the estimation error of $\widehat{c}^{*n}_i$ can be decomposed as
\begin{equation}\label{errors.chat*}
	\beta^n_i \equiv \widehat{c}^{*n}_i-c^n_i = \xi^{n,0}_i + \xi^{n,1}_i + \xi^{n,2}_i + N(p)^n_i + M(p)^n_i
\end{equation}
where
\begin{eqnarray*}
	\xi^{n,0}_i&=&\frac{1}{k_n-l_n}\sum_{h=1}^{k_n-l_n+1}c^n_{i+h}-c^n_i\\
	\xi^{n,1}_i&=&\frac{1}{(k_n-l_n)\Delta_n}\sum_{h=1}^{k_n-l_n+1}D^n_{i+h}\\
	\xi^{n,2}_i&=&\frac{-1}{(k_n-l_n)\Delta_n}\sum_{h=1}^{k_n-l_n+1}R^n_{i+h}\\
	N(p)^n_i&=&\frac{1}{(k_n-l_n)\Delta_n}\Big(\sum_{h=0}^{m(n,p)-1}\zeta(1)^n_{i+b(n,p,h)}+\sum_{h=m(n,p)(p+1)l_n}^{k_n-l_n}\zeta^n_{i+1+h}\Big)\\
	M(p)^n_i&=&\frac{1}{(k_n-l_n)\Delta_n}\sum_{h=0}^{m(n,p)-1}\zeta(p)^n_{i+a(n,p,h)}
\end{eqnarray*}

\subsubsection{bounds on $\|\xi^{n,r}_i\|$}
By assumption \ref{A-v}, (\ref{SA-v}), (\ref{classic})
\begin{equation}\label{est.xi(0)n}
\left\{\begin{array}{lcl}
	\left\| E\left(\xi^{n,0}_{i}|\F^{(0),n}_{i}\right) \right\| &\le& Kk_n\Delta_n\\
	E\left( \|\xi^{n,0}_{i}\|^q|\F^{(0),n}_{i}\right) &\le& K_q(k_n\Delta_n)^{(q/2)\wedge1},\, q\ge0
\end{array}\right.
\end{equation}
combined with (\ref{tuning}),
\begin{equation}\label{est.xi(1)n}
\left\{\begin{array}{lcl}
	\left\| E\left(\xi^{n,1}_i|\F^{(0),n}_i\right) \right\| &\le& K\Delta_n^{1/2}\\
	E\left(\|\xi^{n,1}_i\|^q|\F^{(0),n}_i\right) &\le& K_q\Delta_n^{[(q/2)\wedge1]/2},\,q\in\mathbb{N}^+
\end{array}\right.
\end{equation}
By assumption \ref{A-r},
\begin{equation}\label{est.xi(2)n}
\begin{array}{l}
	\left\| E^n_i\big(\xi^{n,2}_i|\big) \right\| \le K\Delta_n^{-1}\\
	E^n_i\big(\|\xi^{n,2}_i\|^q\big) \le
	\left\{\begin{array}{l}
		K\,k_n^{-1/2},\, q=1;\\
		K_q\left(k_n^{-q+1} + k_n^{-q}\Delta_n^{-q/2+1}\right), q\in\mathbb{N}^+/\{1\}.
	\end{array}\right.
\end{array}
\end{equation}

\subsubsection{estimates of $\zeta(X,p)^n_i$ \& $\zeta(X,p)'^n_i$}
$\widebar{C}^n_i=(\psi^n)^{-1}C^n_{i\Delta_n,(l_n-1)\Delta_n}$ in view of (\ref{def.Un}),
hence by (\ref{SA-v})
\begin{equation}\label{est.Cbar}
	\|\widebar{C}^n_i\| \le K\Delta_n
\end{equation}
According to (\ref{phi.cond}) we have $\widebar{X}^n_i=-\psi_n^{-1/2}\sum_{h=0}^{l_n-1}(\varphi^n_{h+1}-\varphi^n_h)(X^n_{i+h}-X^n_i)$, then by (\ref{classic})
\begin{equation}\label{est.Xbar}
E\left(\|\widebar{X}^n_i\|^q|\F^{(0),n}_i\right)\le K_q\Delta_n^{q/2}
\end{equation}

Adopt the argument for (5.21) in \cite{j09} in the multivariate setting, we have
\begin{equation}\label{est.Wbar}
\begin{array}{lcl}
	E\big(\widebar{W}^n_h\widebar{W}^{n,\T}_{h'}|\F^{(0),n}_i\big) &=& \frac{l_n\Delta_n}{\psi_n}\phi_0\big(\frac{|h'-h|}{l_n}\big)\I + O_p(l_n^{-1/2}\Delta_n)\\
	E\big(\|\widebar{W}^n_h\|^{2m}|\F^{(0),n}_i\big) &=& \Delta_n^m(2m-1)!! + O_p(l_n^{-1}\Delta_n^m),\,\,m\in\mathbb{N}^+
\end{array}
\end{equation}
Let $U^n_i(p)=\sum_{h=i}^{i+pl_n-1}(\sigma^n_i\widebar{W}^n_h)(\sigma^n_i\widebar{W}^n_h)^\T$, $S^n_i(p)=\sum_{h=i}^{i+pl_n-1}\widebar{C}^n_h$, then
\begin{multline}\label{zeta.W.p.2}
	\zeta(W,p)^{n,jk}_i\zeta(W,p)^{n,lm}_i = U^n_i(p)^{jk}U^n_i(p)^{lm} + S^n_i(p)^{jk}S^n_i(p)^{lm}\\ 
	- U^n_i(p)^{jk}S^n_i(p)^{lm} - U^n_i(p)^{lm}S^n_i(p)^{jk}
\end{multline}
By (\ref{sum.phi.quadratic}), (\ref{est.Cbar}), (\ref{est.Wbar}), (\ref{zeta.W.p.2}), and through similar arguments in section 5.3 of \cite{j09} with a modification for multi-dimension, and exploit the connection between $\zeta(W,p)^n_i$ and $\zeta(X,p)^n_i$, in view of (\ref{tensors}), we have the following lemma:
\begin{lem}\label{est.zetaX}
Assume assumption \ref{A-v}, (\ref{SA-v}), $l_n$ satisfies (\ref{tuning}), then
\begin{eqnarray*}
	E\left(\|\zeta(X,p)^n_i\|^4|\F^{(0),n}_i\right)&\le&Kp^4\Delta_n^2\\
	\left\|E\left[\zeta(X,p)^n_i|\F^{(0),n}_i\right]\right\| &\le& Kp\Delta_n\\
	E\left[\zeta(X,p)'^n_i|\F^{(0),n}_i\right] &=& \frac{\theta^2l_n}{\psi_n}(p\Phi_{01}-\Psi_{01})\,c^n_i + p^2\,O_p(\Delta_n^{1/4})\\
	E\left[\zeta(X,p)^{n,jk}_i\zeta(X,p)^{n,lm}_i|\F^{(0),n}_i\right] &=& \frac{2\theta^4}{\psi_n^2}(p\Phi_{00}-\Psi_{00})\,\Sigma(c^n_i)^{jk,lm}\\
	&&\hspace{35mm} +p^2\,O_p(\Delta_n^{5/4})
\end{eqnarray*}
\end{lem}

\subsubsection{estimates of $\zeta(p)^n_i$}
For $i\le h,h'\le i+pl_n-1$, by (\ref{sum.phi.quadratic})
\begin{eqnarray}\label{est.Gamma}
	\Gamma^n_{h,h'} &=& \frac{1}{\psi_nl_n}\phi_1\Big(\frac{|h'-h|}{l_n}\Big)\rr_i+O_p(\Delta_n^{5/4})\\
	\sum_{h=i}^{i+pl_n-2}\,\sum_{h'=h+1}^{i+pl_n-1}\Gamma^{n,jk}_{h,h'}\Gamma^{n,lm}_{h,h'}&=&\frac{1}{\psi_n^2}\left(p\Phi_{11}-\Psi_{11}\right)\gamma^{n,jk}_i\gamma^{n,lm}_i + p^2\,O_p(\Delta_n^{5/4}) \nonumber
\end{eqnarray}

Let $\xi^{n,j_1\cdots j_q}_{h_1\cdots h_q}=\prod_{v=1}^q\big(\widebar{Y}^{n,j_v}_{h_v}-\widebar{X}^{n,j_v}_{h_v}\big)$. By assumption \ref{A-r},
\begin{multline}\label{est.Ybar-Xbar}
	E\left[\big(\widebar{Y}^{n,j}_h-\widebar{X}^{n,j}_h\big)^q\big(\widebar{Y}^{n,k}_{h'}-\widebar{X}^{n,k}_{h'}\big)^r|\widetilde{\F}^n_{h\wedge h'-1}\right] =\\ 
	\left\{\begin{array}{ll}
		0                    & q+r=1\\
		\Gamma^{n,jk}_{h,h'} & q=r=1\\
		O_p\big(l_n^{-7/2}\big)\1{|h-h'|\le l_n} & q+r=3\\
		O_p(l_n^{-8})        & q+r=8
	\end{array}\right.\\
	E\left[\xi^{n,jklm}_{hhh'h'}|\widetilde{\F}^n_{h\wedge h'}\right] = \Gamma^{n,jk}_h\Gamma^{n,lm}_{h'} + \Gamma^{n,jl}_{h,h'}\Gamma^{n,km}_{h,h'} +  \Gamma^{n,jm}_{h,h'}\Gamma^{n,kl}_{h,h'} + O_p\left(l_n^{-5}\right)
\end{multline}
then by (\ref{est.Xbar}), (\ref{est.Ybar-Xbar})
\begin{multline*}
	E(\zeta^n_h|\widetilde{\F}^n_h) = \widebar{X}^n_h{\widebar{X}^n_h}^\T - \widebar{C}^n_h,\,\,\,
	\|E^n_h\left(\zeta^n_h\right)\| \le K\Delta_n^{3/2},\,\,\,
	E^n_h\left(\|\zeta^n_h\|^4\right) \le K\Delta_n^4 \nonumber\\
	E\big(\zeta^{n,jk}_h\zeta^{n,lm}_{h'}|\widetilde{\F}^n_{h\wedge h'}\big) = \sum_{r=1}^3\vartheta(r)^{n,jk,lm}_{h,h'}\\
	+ \big(\widebar{X}^{n,j}_h\widebar{X}^{n,k}_h - \widebar{C}^{n,jk}_h\big)\big(\widebar{X}^{n,l}_{h'}\widebar{X}^{n,m}_{h'} - \widebar{C}^{n,lm}_{h'}\big)
\end{multline*}
where
\begin{eqnarray*}
	\vartheta(1)^{n,jk,lm}_{h,h'}&=&\widebar{X}^{n,j}_h\widebar{X}^{n,l}_{h'}\xi^{n,km}_{hh'}+\widebar{X}^{n,j}_h\widebar{X}^{n,m}_{h'}\xi^{n,kl}_{hh'}\\
	&&\hspace{30mm} +\widebar{X}^{n,k}_h\widebar{X}^{n,l}_{h'}\xi^{n,jm}_{hh'}+\widebar{X}^{n,k}_h\widebar{X}^{n,m}_{h'}\xi^{n,jl}_{hh'}\\
	\vartheta(2)^{n,jk,lm}_{h,h'}&=&\xi^{n,jklm}_{hhh'h'} - \xi^{n,jk}_{hh}\Gamma^{n,lm}_{h'} - \xi^{n,lm}_{h'h'}\Gamma^{n,jk}_h + \Gamma^{n,jk}_h\Gamma^{n,lm}_{h'}\\
	\vartheta(3)^{n,jk,lm}_{h,h'}&=&\widebar{X}^{n,j}_h\xi^{n,klm}_{hh'h'}+\widebar{X}^{n,k}_h\xi^{n,jlm}_{hh'h'}+\widebar{X}^{n,l}_{h'}\xi^{n,jkm}_{hhh'}+\widebar{X}^{n,m}_{h'}\xi^{n,jkl}_{hhh'}
\end{eqnarray*}

Let $\Upsilon^{n,jk,lm}_{h,h'} = \Theta(\widebar{X}^n_h\widebar{X}^{n,\T}_{h'},\Gamma^n_{h,h'})^{jk,lm}$ in light of (\ref{tensors}), then
\begin{multline*}
	E\left(\zeta^{n,jk}_h\zeta^{n,lm}_{h'}|\widetilde{\F}^n_{h\wedge h'}\right) = \big(\widebar{X}^{n,j}_h\widebar{X}^{n,k}_h - \widebar{C}^{n,jk}_h\big)\big(\widebar{X}^{n,l}_{h'}\widebar{X}^{n,m}_{h'} - \widebar{C}^{n,lm}_{h'}\big)\\
	+ \Upsilon^{n,jk,lm}_{h,h'} + \Sigma(\Gamma^n_{h,h'})^{jk,lm} + O_p\big(l_n^{-5} + (\|\widebar{X}^n_h\|+\|\widebar{X}^n_{h'}\|)l_n^{-7/2}\big)
\end{multline*}

hence 
\begin{multline*}
	\zeta(p)^{n,jk}_i\zeta(p)^{n,lm}_i = \zeta(X,p)^{n,jk}_i\zeta(X,p)^{n,lm}_i\\
	+ \sum_{h=i}^{i+pl_n-2}\sum_{h'=h+1}^{i+pl_n-1}\Big[\Upsilon^{n,jk,lm}_{h,h'} + \Upsilon^{n,lm,jk}_{h,h'} + \Sigma(\Gamma^n_{h,h'})^{jk,lm} + \Sigma(\Gamma^n_{h,h'})^{lm,jk}\Big]\\
	+ \sum_{h=i}^{i+pl_n-1}\Big[\Upsilon^{n,jk,lm}_{h,h'} + \Sigma(\Gamma^n_{h,h'})^{jk,lm}\Big] + p^2\,O_p(\Delta_n^{5/4})
\end{multline*}
then by (\ref{est.Xbar}), (\ref{est.Gamma})
\begin{multline*}
	E\left[\zeta(p)^{n,jk}_i\zeta(p)^{n,lm}_i|\widetilde{\F}^n_i\right] = \zeta(X,p)^{n,jk}_i\zeta(X,p)^{n,lm}_i +\\
	\frac{2}{\psi_nl_n}\Theta(\zeta(X,p)'^n_i,\rr^n_i)^{jk,lm} + \frac{2}{\psi_n^2}(p\Phi_{11}-\Psi_{11})\,\Sigma(\rr^n_i)^{jk,lm} + p^2\,O_p(\Delta_n^{5/4})
\end{multline*}
According to these results and lemma \ref{est.zetaX}, one can get the following lemma
\begin{lem}\label{est.zetaY}
Assume assumption \ref{A-v}, \ref{A-r}, (\ref{SA-v}), $l_n$ satisfies (\ref{tuning}), then
\begin{eqnarray*}
	E[\zeta(p)^n_i|\widetilde{\F}^n_i] &=& \zeta(X,p)^n_i \\
	\|E^n_i[\zeta(p)^n_i]\| &\le& Kp\Delta_n\\
	E^n_i(\|\zeta(p)^n_i\|^q) &\le& K_q\,p^{\lfloor q/2\rfloor\vee1}\Delta_n^{q/2},\hspace{2mm} q=1,2,3,4
\end{eqnarray*}
moreover
\begin{equation*}
	\left|E^n_i\left[\zeta(p)^{n,jk}_i\zeta(p)^{n,lm}_i\right] - (p+1)\theta\Delta_n\,\Xi(c^n_i,\gamma^n_i; p)^{jk,lm}\right| \le K\Delta_n^{5/4}
\end{equation*}
where
\begin{multline}\label{def.Xi(p)}
	\Xi(x,z;p) = \frac{2\theta}{\phi_0(0)^2}\left[\frac{p\Phi_{00}-\Psi_{00}}{p+1}\Sigma(x)\right.\\ 
	\left.+ \frac{p\Phi_{01}-\Psi_{01}}{\theta^2(p+1)}\Theta(x,z) + \frac{p\Phi_{11}-\Psi_{11}}{\theta^4(p+1)}\Sigma(z)\right]
\end{multline}
\end{lem}

Let $p\asymp\Delta_n^{-1/12}$, based on lemma \ref{est.zetaY},
\begin{eqnarray}\label{est.M(p).N(p)}
	\left\|E^n_i\left(M(p)^n_i\right) \right\| &\le& K\Delta_n^{1/2} \nonumber\\
	\left\|E^n_i\left(N(p)^n_i\right) \right\| &\le& K p^{-1}\Delta_n^{1/2} \nonumber\\
	E^n_i\left(\|M(p)^n_i\|^q\right) &\le&
	\left\{\begin{array}{ll}
		K_q\big(k_n\Delta_n^{1/2}\big)^{-q/2}, & q=1,2,4\\
		K\big(k_n\Delta_n^{1/2}\big)^{-2}, & q=3
	\end{array}\right.\\
	E^n_i\left(\|N(p)^n_i\|^q|\right) &\le&
	\left\{\begin{array}{ll}
		K p(k_n\Delta_n^{1/2})^{-1} & q=1\\
		K_q p^{-q/2}(k_n\Delta_n^{1/2})^{-q/2} & q=2,4\\
		K p^{-1}(k_n\Delta_n^{1/2}\big)^{-2} & q=3
	\end{array}\right. \nonumber
\end{eqnarray}

\subsubsection{estimates of $\beta^n_i$}
We need to define more variables:
\begin{equation}\label{def.betavars}
\begin{array}{ll}
	\zeta(p)^n_{i,h} = \zeta(p)^n_{i+a(n,p,h)} & A(p)^n_{i+v} = \sum_{h=v}^{v+pl_n-1}(c^n_{i+h}-c^n_i)\Delta_n\\
	D(p)^n_{i+v} = \sum_{h=v}^{v+pl_n-1}D^n_{i+h} & R(p)^n_{i+v} = \sum_{h=v}^{v+pl_n-1}R^n_{i+h}
\end{array}
\end{equation}
we have
\begin{table}[H]
\centering
\caption{Estimates of ingredients}\label{ingredient.estimates}
\begin{tabular}{c|ll}
	scaling properties & $E(\|\cdot\|^2|\F^n_i)$ & $\|E(\cdot|\F^n_i)\|$\\
	\hline
	$R(p)^n_h$         & $p\Delta_n^{3/2}$       & $p\Delta_n^{3/2}$\\
	$D(p)^n_h$         & $p\Delta_n^{3/2}$       & $p\Delta_n$ \\
	$A(p)^n_{i+v}$     & $p^2\Delta_n^2\big(p\Delta_n^{-1/2}+v\big)$ & $p\Delta_n^{3/2}\big(p\Delta_n^{-1/2}+v\big)$\\
	$\zeta(p)^n_{i,h}$   & $p\Delta_n$           & $p\Delta_n$\\
\end{tabular}
\end{table}
Define 
\begin{multline*}
	\alpha(p)^n_{i,h} = - R(p+1)^n_{i+a(n,p,h)} + D(p+1)^n_{i+a(n,p,h)}\\
	 + A(p+1)^n_{i+a(n,p,h)} + \zeta(p+1)^n_{i,h}
\end{multline*}
By table \ref{ingredient.estimates} and Cauchy-Schwarz inequality,
\begin{multline}\label{a(p)2-zeta(p)2}
	E^n_i\left(\left|\alpha(p)^{n,jk}_{i,h}\alpha(p)^{n,lm}_{i,h}-\zeta(p+1)^{n,jk}_{i,h}\zeta(p+1)^{n,lm}_{i,h}\right|\right)\\
\le K\big(p^2\Delta_n^{5/4}+p^{3/2}\Delta_n^{3/2}v^{1/2}\big)
\end{multline}

Given $j,k,l,m=1,\cdots,d$, by table \ref{ingredient.estimates} we have
\begin{equation*}
	\left| E^n_i\big(\beta^{n,jk}_i\beta^{n,lm}_i\big) - (k_n\Delta_n^{1/2})^{-1}\Xi(c^n_i,\rr_i)^{jk,lm} \right| = \sum_{r=1}^5\mu^{n,r}_i + pO_p((k_n^2\Delta_n)^{-1})
\end{equation*}
where
\begin{eqnarray*}
	\mu^{n,1}_i &=&
	\frac{1}{k_n^2\Delta_n^2}\sum_{h=0}^{m(n,p)-1}E^n_i\left(\left|\alpha(p)^{n,jk}_{i,h}\alpha(p)^{n,lm}_{i,h}
	- \zeta(p+1)^{n,jk}_{i,h}\zeta(p+1)^{n,lm}_{i,h}\right|\right)\\
	\mu^{n,2}_i &=&
	\frac{1}{k_n^2\Delta_n^2}\sum_{h=0}^{m(n,p)-2}\sum_{h'=h+1}^{m(n,p)-1}\left|E^n_i\left[\alpha(p)^{n,jk}_{i,h}\alpha(p)^{n,lm}_{i,h'}
	+\alpha(p)^{n,lm}_{i,h}\alpha(p)^{n,jk}_{i,h'}\right]\right|\\
	\mu^{n,3}_i &=& \frac{1}{k_n^2\Delta_n^2}\sum_{h=0}^{m(n,p)-1}E^n_i\left(\left|\zeta(p+1)^{n,jk}_{i,h}\zeta(p+1)^{n,lm}_{i,h}\right.\right.\\
	&&\hspace{24mm} \left.\left. -(p+2)\theta\Delta_n\,\Xi\big(c^n_{i+a(n,p,h)},\rr^n_{i+a(n,p,h)}; p+1\big)^{jk,lm}\right|\right)\\
	\mu^{n,4}_i &=& \frac{(p+2)\theta}{k_n^2\Delta_n}\sum_{h=0}^{m(n,p)-1}\left|E^n_i\left[\Xi\big(c^n_{i+a(n,p,h)},\rr^n_{i+a(n,p,h)}; p+1\big)^{jk,lm}\right.\right. \\
	&&\hspace{66mm} \left.\left. - \Xi(c^n_i,\rr^n_i; p+1)^{jk,lm}\right]\right|
\end{eqnarray*}
\begin{eqnarray*}
	\mu^{n,5}_i &=& \frac{1}{k_n\Delta_n^{1/2}}\left|\frac{(p+2)\theta}{k_n\Delta_n^{1/2}}\left\lfloor\frac{k_n}{(p+1)l_n}\right\rfloor-1\right|\left|\Xi(c^n_i,\rr^n_i; p+1)^{jk,lm}\right|\\
	&&\hspace{27mm} + \frac{1}{k_n\Delta_n^{1/2}}\left|\Xi(c^n_i, \rr^n_i; p+1)^{jk,lm} - \Xi(c^n_i,\rr^n_i)^{jk,lm}\right|
\end{eqnarray*}
Use table \ref{ingredient.estimates} and (\ref{a(p)2-zeta(p)2}) to get bounds on $\mu^{n,r}_i,\,r=1,2,3,4,5$; combine (\ref{errors.chat*}), (\ref{est.xi(0)n}), (\ref{est.xi(1)n}), (\ref{est.xi(2)n}), (\ref{est.M(p).N(p)}), we get the following lemma:

\begin{lem}\label{est.beta} Assume (\ref{tuning}), (\ref{SA-v}) and assumption \ref{A-v}, \ref{A-r}, given $p\in\mathbb{N}^+$,
\begin{eqnarray*}
	\|E^n_i(\beta^n_i)\| &\le& K k_n\Delta_n\\
	E^n_i(\|\beta^n_i\|^q) &\le& 
	\left\{\begin{array}{ll}
		K_q\left[(k_n\Delta_n)^{(q/2)\wedge1} + (k_n\Delta_n^{1/2})^{-q/2}\right],& q = 1,2,4\\
		Kk_n\Delta_n, & q = 3
	\end{array}\right.
\end{eqnarray*}
additionally
\begin{multline*}
	\left|E^n_i\big(\beta^{n,jk}_i\beta^{n,lm}_i\big) - (k_n\Delta_n)^{-1/2}\Xi(c^n_i,\gamma^n_i)^{jk,lm}\right|\\
	\le K\big[k_n\Delta_n + p^{-1}(k_n\Delta_n^{1/2})^{-1}\big]
\end{multline*}
\end{lem}

\subsection{Structure}
Define
\begin{equation*}
\begin{array}{rcl}
	\lambda(x,z) &=& \sum^d_{j,k,l,m=1}\partial^2_{jk,lm}g(x)\times\Xi(x,z)^{jk,lm}\\
	\eta^n_i &=& \lambda(\widehat{c}^n_i,\widehat{\gamma}^n_i)-\lambda(\widehat{c}^{*n}_i,\widehat{\gamma}^n_i)
\end{array}
\end{equation*}
As $\Delta_n^{-1/4}|a^n_t-1|<k_n\Delta_n^{3/4}\to0$, letting $a^n_t=1$ 
doesn't not affect the asymptotic analysis. By Cram\'er-Wold theorem, we can suppose $g$ is $\R$-valued. We have
\begin{equation}\label{decomp}
	\Delta_n^{-1/4}\big[\widehat{S}(g)^n-S(g)\big] = \widebar{S}^{n,0} + \widebar{S}^{n,1} + \widebar{S}(p)^{n,2} + \widebar{S}^{n,3} + \widebar{S}(p)^{n,4}
\end{equation}
where
\begin{eqnarray*}
	\widebar{S}^{n,0}_t&=&\Delta_n^{-1/4}\Bigg[\sum_{i=0}^{N^n_t-1}\int_{ik_n\Delta_n}^{(i+1)k_n\Delta_n}g(c^n_{ik_n})-g(c_s)\ds s-\int_{N^n_tk_n\Delta_n}^tg(c_s)\ds s\Bigg]\\
	\widebar{S}^{n,1}_t&=&k_n\Delta_n^{3/4}\sum_{i=0}^{N^n_t-1}\Big[g(\widehat{c}^n_{ik_n})-g(\widehat{c}^{*n}_{ik_n}) - (2k_n\Delta_n^{1/2})^{-1}\eta^n_{ik_n}\Big]\\
	\widebar{S}(p)^{n,2}_t&=&k_n\Delta_n^{3/4}\sum_{i=0}^{N^n_t-1}\sum^d_{j,k=1}\partial_{jk}g(c^n_{ik_n})\Bigg[\sum_{r=0}^2\xi(r)^{n,jk}_{ik_n}+N(p)^{n,jk}_{ik_n}\Bigg]
\end{eqnarray*}
\begin{eqnarray*}
	\widebar{S}^{n,3}_t&=&k_n\Delta_n^{3/4}\sum_{i=0}^{N^n_t-1}\Big[g(\widehat{c}^{*n}_{ik_n})-g(c^n_{ik_n})-\sum^d_{j,k=1}\partial_{jk}g(c^n_{ik_n}) \beta^{n,jk}_{ik_n}\\
	&&\hspace{55mm} - (2k_n\Delta_n^{1/2})^{-1}\lambda(\widehat{c}^{*n}_i,\widehat{\rr}^n_i)\Big]\\
	\widebar{S}(p)^{n,4}_t&=&k_n\Delta_n^{3/4}\sum_{i=0}^{N^n_t-1}\sum^d_{j,k=1}\partial_{jk}g(c^n_{ik_n})\times M(p)^{n,jk}_{ik_n}
\end{eqnarray*}

\subsection{Asymptotic negligibility}
First of all, we need to get bounds on $\partial^r g(\widehat{c}^n_i),\,r\le3$, where $\partial^rg$ denotes the $r$-th order partial derivatives. Let $\widebar{c}^n_i = (k_n-l_n)^{-1}\sum_{h=1}^{k_n-l_n+1}c^n_i$ and $I^n_t=\{0,1,\cdots,N^n_t-1\}$, note that $|I^n_t|\asymp (k_n\Delta_n)^{-1}$, according to lemma \ref{est.chat-chat*}, there is a sequence $a_n\to0$ such that
\begin{equation*}
	\E\Big(\sup_{i\in I_n}\|\widehat{c}^n_i-\widehat{c}^{*n}_i\|\Big) \le K\left(a_n\Delta_n^{\kappa-1/2-(\rho-1/4)\nu - (1-2\rho)} + \Delta_n^{\kappa-1/2}\right)
\end{equation*}
Note $\widehat{c}^{*n}_i-\widebar{c}^{n}_i = \xi^{n,1}_i + \xi^{n,2}_i + N(p)^n_i + M(p)^n_i$, by (\ref{est.xi(1)n}), (\ref{est.xi(2)n}), (\ref{est.M(p).N(p)}) and $\kappa<3/4$, $E^n_i\big(\|\widehat{c}^{*n}_i-\widebar{c}^{n}_i\|^4\big)\le K\Delta_n^{2\kappa-1}$, so
\begin{equation*}
	\E\Big(\sup_{i\in I_n}\|\widehat{c}^{*n}_i-\widebar{c}^n_i\|\Big) \le K\Delta_n^{3\kappa-2}
\end{equation*}
hence by (\ref{tuning}) and Markov's inequality
\begin{equation}\label{chat-c.uniform}
\sup_{i\in I_n} \|\widehat{c}^n_i-\widebar{c}^n_i\| = o_p(1)
\end{equation}
According to (\ref{SA-v}) and convexity, $\widebar{c}^n_i\in\mathcal{S}$. By (\ref{chat-c.uniform}) $\widehat{c}^n_i\in\mathcal{S}^\epsilon$ if $n$ is sufficiently large. Therefore by (\ref{g.cond}), in asymptotic analysis we can assume
\begin{equation}\label{g.cond.strong}
\|\partial^rg(\widehat{c}^n_i)\| \le K, \hspace{5mm} \forall r=0,1,2,3,\, \forall i\in I^n_t
\end{equation}

Through an almost identical argument for lemma 4.4 in \cite{jr13},
\begin{equation}\label{v0}
	\widebar{S}^{n,0}\overset{u.c.p.}{\longrightarrow}0
\end{equation}

Define function $g_n$ on $\mathcal{S}^+_d\times\mathcal{S}^+_d$ as $g_n(x,z) = g(x) - (k_n\Delta_n^{1/2})^{-1}\xi(x,z)$
according to (\ref{SA-v}),
\begin{multline*}
	\|g_n(x,z)-g_n(y,z)\| \le K\|x-y\|\\
	+ K(k_n\Delta_n^{1/2})^{-1}\|x-y\|\left(\|x\|^2+\|z\|^2 + \|x-y\|^2+\|z\|\|x-y\|\right)
\end{multline*}
so $\|g_n(x,z)-g_n(y,z)\| \le K\|x-y\|$ when $n$ is sufficiently large. By lemma \ref{est.chat-chat*}
\begin{multline*}
	\E\left(\sup_{s\in[0,t]}\big\|\widebar{S}^{n,1}_s\big\|\right) \le k_n\Delta_n^{3/4}\sum_{i=0}^{N^n_t-1}\left\|g_n(\widehat{c}^n_{ik_n},\widehat{\gamma}^n_{ik_n})-g_n(\widehat{c}^{*n}_{ik_n},\widehat{\gamma}^n_{ik_n})\right\|\\
	\le Kt\left(a_n\Delta_n^{1/4-(\rho-1/4)\nu-(1-2\rho)} + \Delta_n^{1/4}\right)
\end{multline*}
Since $\rho>\frac{3-\nu}{4(2-\nu)}$, $1/4-(\rho-1/4)\nu-(1-2\rho)>0$, we have the following lemma:
\begin{lem}\label{lemma.v1}
	Assume assumption \ref{A-v}, \ref{A-r}, (\ref{g.cond.strong}), (\ref{tuning}) then
	$$\widebar{S}^{n,1}\overset{u.c.p.}{\longrightarrow}0$$
\end{lem}

Given $e^n_i\in\R^{d\times d}$, consider the process
\begin{equation*}
	\widetilde{S}^n_t = k_n\Delta_n^{3/4} \sum_{i=0}^{N^n_t-1} \sum^d_{j,k=1}\partial_{jk}g(c^n_{ik_n})\times e^{n,jk}_{ik_n}
\end{equation*}
suppose $e^n_i$ satisfies
\begin{equation}\label{xi.cond}
	\left\{\begin{array}{lcl}
	\left\| E^n_i\left(e^n_{i}\right) \right\| &\le& K\Delta_n^{1/4} a_n \\
	E^n_i\left( \|e^n_{i}\|^2\right) &\le& K(k_n\Delta_n^{1/2})^{-1} b_n
\end{array}\right.
\end{equation}
where $a_n,b_n\to0$. Since $\partial g$ is bounded by (\ref{SA-v}),
\begin{multline*}
	\E\left(\sup_{s\in[0,t]} \big\|\widetilde{S}^n_s\big\|\right) \le Kk_n\Delta_n^{3/4}\sum_{i=0}^{N^n_t-1}\E\left(\left\|E^n_{ik_n}\left(e^n_{ik_n}\right)\right\|\right) \\ + Kk_n\Delta_n^{3/4}\E\left(\sup_{s\in[0,t]}\left\|\sum_{i=0}^{N^n_s-1}\left[ e^n_{ik_n} - E^n_{ik_n}\left(e^n_{ik_n}\right) \right]\right\|\right)
\end{multline*}
by lemma \ref{Doob.max.ineq},
\begin{equation*}
	\E\left(\sup_{s\in[0,t]}\left\|\sum_{i=0}^{N^n_s-1}\left[ e^n_{ik_n} - E^n_{ik_n}\left(e^n_{ik_n}\right) \right]\right\|\right)
	\le K\left(\sum_{i=0}^{N^n_t-1}\E\left(\|e^n_{ik_n}\|^2\right)\right)^{1/2}
\end{equation*}
note that $k_n\Delta_nN^n_t\asymp t$, we have
\begin{equation*}
	\E\left(\sup_{s\in[0,t]} \big\|\widetilde{S}^n_s\big\|\right) \le K\left(ta_n + \sqrt{tb_n}\right)\to0
\end{equation*}
To show the asymptotic negligibility of $\widebar{S}^{n,2}$, we need to show $\xi_i$ satisfies (\ref{xi.cond}) in each of the following 4 cases:
\begin{itemize}
	\item[(i)] when $e^n_i=\xi^{n,0}_i$, 
	by (\ref{est.xi(0)n}), $a_n=k_n\Delta_n^{3/4},\,b_n=(k_n\Delta_n^{3/4})^2$;
	\item[(ii)] when $e^n_i=\xi^{n,1}_i$,
	by (\ref{est.xi(1)n}), $a_n=\Delta_n^{1/4}$, $b_n=k_n\Delta_n$;
	\item[(iii)] when $e^n_i=\xi^{n,2}_i$,
	by (\ref{est.xi(2)n}), $a_n=\Delta_n^{3/4}$, $b_n=\Delta_n^{1/2}$;
	\item[(iv)] when $e^n_i=N(p)^n_i$, by (\ref{est.M(p).N(p)}), $a_n=p^{-1}\Delta_n^{1/4}$, $b_n=p^{-1}$.
\end{itemize}
Hence we have the following lemma:
\begin{lem}\label{lemma.v2}
	Assume assumption \ref{A-v}, \ref{A-r}, (\ref{g.cond.strong}), (\ref{tuning}), and let $p\asymp\Delta_n^{-1/12}$, then
	$$\widebar{S}(p)^{n,2}\overset{u.c.p.}{\longrightarrow}0$$
\end{lem}

Let $\chi^n_i = \widehat{\rr}^n_i - \rr^n_i$, by (\ref{classic}), the choice of $m_n$ and Jensen's inequality
\begin{equation}\label{est.chi}
	E^n_i(\|\chi^n_i\|^q) \le K_q\Delta_n^{q/4},\, q=1,2
\end{equation}
Let $\eta'^n_i=\lambda(\widehat{c}^{*n}_i,\widehat{\gamma}^n_i)-\lambda(c^n_i,\gamma^n_i)$, then by (\ref{tensors}), (\ref{def.Xi})
\begin{equation*}
	\|\eta'^n_i\| \le K\left(\|\beta^n_i\| + \|\chi^n_i\| + \|\beta^n_i\|^2 + \|\beta^n_i\|\|\chi^n_i\| + \|\chi^n_i\|^2\right)
\end{equation*}
hence by lemma \ref{est.beta}, (\ref{est.chi}), (\ref{tuning})
\begin{equation}\label{est.eta}
	E^n_i(\|\eta'^n_i\|^q) \le K_q(k_n\Delta_n^{1/2})^{-q/2},\hspace{2mm} q=1,2
\end{equation}

We can rewrite $\widebar{S}^{n,3}$ as
\[\widebar{S}^{n,3}=G^n+H^n\]
where
\begin{eqnarray*}
	G^n_t&=&k_n\Delta_n^{3/4}\sum_{i=0}^{N^n_t-1}\left[s^n_{ik_n}+u^n_{ik_n}+E^n_{ik_n}\left(v^n_{ik_n}\right)\right]\\
	H^n_t&=&k_n\Delta_n^{3/4}\sum_{i=0}^{N^n_t-1}\left[v^n_{ik_n}-E^n_{ik_n}\left(v^n_{ik_n}\right)\right]
\end{eqnarray*}
\begin{eqnarray*}
	s^n_i &=& g(c^n_i+\beta^n_i)-g(c^n_i)-\sum^d_{j,k=1}\partial_{jk}g(c^n_i)\beta_i^{n,jk}\\ 
	&&\hspace{45mm} - \frac{1}{2}\sum^d_{j,k,l,m=1}\partial^2_{jk,lm}g(c^n_i)\beta_i^{n,jk}\beta_i^{n,lm}\\
	u^n_i&=&\frac{1}{2k_n\Delta_n^{1/2}}\left[\xi(c^n_i,\rr^n_i) - \xi(\widehat{c}^{*n}_i,\widehat{\rr}^n_i)\right]\\
	v^n_i&=&\frac{1}{2}\sum^d_{j,k,l,m=1}\partial^2_{jk,lm}g(c^n_i)\left[\beta_i^{n,jk}\beta_i^{n,lm} - (k_n\Delta_n^{1/2})^{-1}\Xi(c^n_i,\rr^n_i)^{jk,lm}\right]
\end{eqnarray*}

By lemma \ref{est.beta}, (\ref{g.cond.strong}), (\ref{est.eta}), if we let $p\asymp\Delta_n^{-12}$, 
\begin{equation}\label{est.G}
	\E\left(\sup_{s\in[0,t]}\|G^n_s\|\right) \le Kt\left[k_n\Delta_n^{3/4} + (k_n\Delta_n^{2/3})^{-1}\right]
\end{equation}
and
\begin{multline*}
	E^n_i(\|v^n_i\|^2) \le K\sum^d_{j,k,l,m=1} E^n_i\left(\left|\beta_i^{n,jk}\beta_i^{n,lm} - \frac{1}{k_n\Delta_n^{1/2}}\Xi^{n,jk,lm}_i\right|^2\right)\\
	\le K\left[k_n\Delta_n + (k_n\Delta_n^{1/2})^{-2}\right]
\end{multline*}
then lemma \ref{Doob.max.ineq} implies
\begin{multline}\label{est.H}
	\E\left(\sup_{s\in[0,t]}\|H^n_s\|\right)\le Kk_n\Delta_n^{3/4}\left(\sum_{i=0}^{N^n_t-1}\E(\|v^n_{ik_n}\|^2)\right)^{1/2}\\
	\le K\sqrt{t}\left[k_n\Delta_n^{3/4} + (k_n\Delta_n^{1/2})^{-1/2}\right] 
\end{multline}

According to (\ref{tuning}), (\ref{est.G}), (\ref{est.H}), we have the following lemma:
\begin{lem}\label{lemma.v3}
	Assume assumption \ref{A-v}, \ref{A-r}, (\ref{g.cond.strong}), (\ref{tuning}) then
	$$\widebar{S}^{n,3}\overset{u.c.p.}{\longrightarrow}0$$
\end{lem}

\subsection{Stable convergence in law to a continuous It\^o semimartingale}
Recall (\ref{def.betavars}), we can write 
\begin{equation*}
	\widebar{S}(p)^{n,4}_t=\frac{k_n}{k_n-l_n}\sum^d_{j,k=1}\Delta_n^{-1/4}\sum^{N^n_t-1}_{i=0}\sum_{h=0}^{m(n,p)-1}\zeta(p)^{n,jk}_{ik_n,h}\times\partial_{jk}g(c^n_{ik_n})
\end{equation*}

Let $\mathcal{H}(p)^n_{i,h}=\F^n_{ik_n+a(n,p,h)}$, by lemma \ref{est.zetaY},
\begin{equation*}
	\Delta_n^{-1/2}\sum^{N^n_t-1}_{i=0}\sum_{h=0}^{m(n,p)-1}\left\|E\big[\zeta(p)^n_{ik_n,h}|\mathcal{H}(p)^n_{i,h}\big]\right\|^2 \le Ktp\Delta_n
\end{equation*}

Let $\Lambda(p)^n_{i,h}=\partial g(c^n_{ik_n})\zeta(p)^n_{ik_n,h}$, $N$ is a bounded martingale orthogonal to $W$ or $N=W^l$ for some $l=1,\cdots,d'$, and $\Delta N(p)^n_{i,h}=N^n_{ik_n+b(n,p,h)}-N^n_{ik_n+a(n,p,h)}$. The following 4 statements about convergence in probability for any indices $j,k,l,m$ can verify the conditions of theorem IX.7.28 in \cite{js03}:
\begin{eqnarray}
	\Delta_n^{-1/4}\sum_{i=0}^{N^n_t-1}\sum_{h=0}^{m(n,p)-1}\left\|E\left[\Lambda(p)^n_{i,h}|\mathcal{H}(p)^n_{i,h}\right]\right\|&\overset{\mathbb{P}}{\longrightarrow}&0\label{V4.c1}\\
	\Delta_n^{-1}\sum_{i=0}^{N^n_t-1}\sum_{h=0}^{m(n,p)-1}E\left[\left\|\Lambda(p)^n_{i,h}\right\|^4|\mathcal{H}(p)^n_{i,h}\right]&\overset{\mathbb{P}}{\longrightarrow}&0\label{V4.c2}\\
	\Delta_n^{-1/4}\sum_{i=0}^{N^n_t-1}\sum_{h=0}^{m(n,p)-1}\left\|E\big[\Lambda(p)^n_{i,h}\Delta N(p)^n_{i,h}|\mathcal{H}(p)^n_{i,h}\big]\right\| &\overset{\mathbb{P}}{\longrightarrow}&0\label{V4.c3}
\end{eqnarray}
\begin{multline}\label{V4.c4}
	\Delta_n^{-1/2}\sum_{i=0}^{N^n_t-1}\sum_{h=0}^{m(n,p)-1}\partial_{jk}g(c^n_{ik_n})\partial_{lm}g(c^n_{ik_n})^\T E\left[\zeta(p)^{n,jk}_{ik_n,h}\zeta(p)^{n,lm}_{ik_n,h}|\mathcal{H}(p)^n_{i,h}\right]\\
	\overset{\mathbb{P}}{\longrightarrow} \int_0^t\partial_{jk}g(c_s)\partial_{lm}g(c_s)^\T\,\Xi(c_s, \gamma_s; p)^{jk,lm}\,\mathrm{d}s
\end{multline}

Under (\ref{g.cond.strong}), one can verify (\ref{V4.c1}), (\ref{V4.c2}) by the second and third claims of lemma \ref{est.zetaY}, respectively. The same argument as that for (5.58) in \cite{j09} leads to (\ref{V4.c3}). By the last claim of lemma \ref{est.zetaY}, the left-hand side of (\ref{V4.c4}) equals
\begin{multline*}
	\sum_{i=0}^{N^n_t-1}\sum_{h=0}^{m(n,p)-1}\partial_{jk}g(c^n_{ik_n})\partial_{lm}g(c^n_{ik_n})^\T\times\\ \Xi\big(c^n_{ik_n+a(n,p,h)}, \rr^n_{ik_n+a(n,p,h)}; p\big)^{jk,lm}(p+1)l_n\Delta_n + tp\,O_p(\Delta_n^{1/4})
\end{multline*}
then (\ref{V4.c3}) is verified by Riemann summation. By theorem IX.7.28 in \cite{js03} we have the following lemma:

\begin{lem}\label{lemma.v4}
	Assume assumption \ref{A-v}, \ref{A-r}, (\ref{g.cond.strong}), (\ref{tuning}), then for $\forall p\in\mathbb{N}^+$,
	\[\widebar{S}(p)^{n,4}\overset{\mathcal{L}-s(f)}{\longrightarrow}Z(p)\]
	where $Z(p)$ is a process defined on an extension of the space  $\left(\Omega,\mathcal{F},(\mathcal{F}_t),\mathbb{P}\right)$,
	such that conditioning on $\mathcal{F}$ it is a  mean-0 continuous It\^o martingale with variance
	\[\widetilde{E}[Z(p)Z(p)^\T|\F]=\int_0^t\sum^d_{j,k,l,m=1}\partial_{jk}g(c_s)\partial_{lm}g(c_s)^\T\,\Xi(c_s, \gamma_s; p)^{jk,lm}\ds s\]
	where $\widetilde{E}$ is the conditional expectation operator on the extended probability space and $\Xi(x,z;p)$ is defined in (\ref{def.Xi(p)}).
\end{lem}

By (\ref{v0}), lemma \ref{lemma.v1}, \ref{lemma.v2}, \ref{lemma.v3}, \ref{lemma.v4}, and $\Xi(x,z;p)\to\Xi(x,z)$ as $p\to\infty$, we arrive at the asymptotic result in section \ref{sec:asymp}.


\end{document}